\documentclass[11pt]{article}
\usepackage{amsmath,amsthm,amsfonts,amssymb,amsopn,amscd}

\def\Ad{{\hbox{\rm Ad}}}

\def\de{\delta}

\def\l{\lambda}

\def\setminus{\smallsetminus}
\def\gA{{\mathfrak A}}

\def\A{{\cal A}}
\def\B{{\cal B}}
\def\C{{\cal C}}

\def\I{{\cal I}}
\def\H{{\cal H}}

\def\S{{\cal S}}

\def\emptyset{\varnothing}

\def\Diff{{\rm Diff}}

\def\Sn{S^{1(n)}}
\def\S2{S^{1(2)}}
\newtheorem{theorem}{Theorem}[section]
\newtheorem{lemma}[theorem]{Lemma}

\newtheorem{corollary}[theorem]{Corollary}

\newtheorem{proposition}[theorem]{Proposition}

\theoremstyle{definition} \newtheorem{definition}[theorem]{Definition}

\theoremstyle{remark} \newtheorem{remark}[theorem]{Remark}

\def\emptyset{\varnothing}
\def\setminus{\smallsetminus}

\def\Diff{{\mathrm {Diff}}}

\def\RR{{\mathbb R}}
\def\CC{{\mathbb C}}
\def\NN{{\mathbb N}}
\def\ZZ{{\mathbb Z}}

\def\sl2{{{\rm SL}(2,\RR)}}
\def\psl2{{{\rm PSL}(2,\RR)}}

\def\u1{{{\rm U}(1)}}
\def\su2{{{\rm SU}(2)}}

\def\so3{{{\rm SO}(3)}}

\def\A{{\mathcal A}}
\def\B{{\mathcal B}}
\def\C{{\mathcal C}}

\def\H{{\mathcal H}}
\def\I{{\mathcal I}}

\def\T{{\mathcal T}}

\title{{\Huge Spectral Triples and \\ the Super-Virasoro Algebra}}

\author{
{\sc Sebastiano Carpi}\footnote{Supported in part by PRIN-MIUR, GNAMPA-INDAM and EU network ``Noncommutative Geometry" 
MRTN-CT-2006-0031962}\\
Dipartimento di Scienze,
Universit\`a di Chieti-Pescara ``G. d'Annunzio''\\
Viale Pindaro, 42, I-65127 Pescara, Italy\\
E-mail: {\tt carpi@sci.unich.it}\\
\phantom{X}\\
{\sc Robin Hillier}$^*$\footnote{Supported by the Gottlieb Daimler- und Karl Benz-Stiftung with a one year research scholarship}
\\
Dipartimento di Matematica,
Universit\`a di Roma ``Tor Vergata'',\\
Via della Ricerca Scientifica, 1, I-00133 Roma, Italy\\
E-mail: {\tt hillier@mat.uniroma2.it}\\
\phantom{X}\\
{\sc Yasuyuki Kawahigashi}\footnote{Supported in part by the
Grants-in-Aid for Scientific Research, JSPS.}\\
Department of Mathematical Sciences\\
University of Tokyo, Komaba, Tokyo, 153-8914, Japan\\
E-mail: {\tt yasuyuki@ms.u-tokyo.ac.jp}\\
\phantom{X}\\
{\sc Roberto Longo}$^*$
\\
Dipartimento di Matematica,
Universit\`a di Roma ``Tor Vergata'',\\
Via della Ricerca Scientifica, 1, I-00133 Roma, Italy\\
E-mail: {\tt longo@mat.uniroma2.it}}

\begin{document}

\date{May 15, 2009}
\maketitle

\begin{abstract}
We construct infinite dimensional spectral triples associated with 
representations of the super-Virasoro algebra. In particular the irreducible, 
unitary positive energy representation of the Ramond algebra with central 
charge $c$ and minimal lowest weight $h=c/24$ is graded and gives rise to 
a net of even $\theta$-summable spectral 
triples with non-zero Fredholm index. The irreducible unitary positive 
energy representations of the Neveu-Schwarz algebra give rise to nets of 
even $\theta$-summable generalised spectral triples where there is no Dirac 
operator but only a superderivation.
\end{abstract}

\bigskip
\bigskip

\thanks{\footnotesize{Supported in part by the ERC Advanced Grant 227458  
OACFT ``Operator Algebras and Conformal Field Theory"}}

\newpage

{\small\tableofcontents}

\section{Introduction}
\label{Intro}
In this paper we make a vital step in the ``noncommutative geometrization" program for Conformal Field Theory, that is in the search of noncommutative geometric invariants associated with conformal nets and their representations.

As we are here in the framework of quantum systems with infinitely many degrees of freedom, natural objects to look for are  
spectral triples in the sense of Connes and Kasparov, see \cite{C}. While there are 
important situations where these objects enter in Quantum Field Theory, 
see e.g. \cite{C, JLW}, the novelty of our work is that our spectral triple depends on 
the sector with respect to the vacuum representation, according to what was proposed by the QFT 
index theorem \cite{L4}.

Let us briefly explain the root of our work. A fundamental object in Connes' 
Noncommutative Geometry is a spectral triple, a noncommutative extension of 
the concept of elliptic pseudo-differential operator, say of the Dirac 
operator. Basically, a (graded) spectral triple 
$({\mathfrak A}, \H, Q)$ consists of a $\ZZ_2$-graded algebra ${\mathfrak A}$ acting on a 
$\ZZ_2$-graded Hilbert 
space $\H$ and an odd selfadjoint linear operator $Q$ on $\H$, with certain 
spectral summability properties and bounded graded commutator with elements 
of ${\mathfrak A}$. 
A spectral triple gives rise to a cyclic cocycle on $A$, the Chern character, 
that evaluates on $K_0$-theory elements of the even part ${\mathfrak A}_+$ 
of ${\mathfrak A}$.

In the present infinite-dimensional case, the right summability condition is 
the trace-class property of the heat kernel, Tr$(e^{-\beta Q^2}) <\infty$, 
$\beta>0$. The involved cohomology is entire cyclic cohomology \cite{C} and 
the corresponding Chern character is given by the Jaffe-Lesniewski-Osterwalder 
formula \cite{JLO} (see also \cite{Connes91,GS}).

Concerning Quantum Field Theory, one expects a natural occurrence of spectral triples in the 
supersymmetric frame. 
We recall a related QFT index theorem for certain massive models on the cylinder in the 
vacuum representation \cite{JLW}.

As explained in \cite{L4}, one may aim for a QFT index theorem, a noncommutative analog of 
the Atiyah-Singer index theorem for systems with infinitely many degrees of freedom, where a 
Doplicher-Haag-Roberts representation (superselection sector, see \cite{H}) represents the 
analog of an elliptic operator. While the operator algebraic and analytic structure behind 
the DHR theory is well understood, in particular by Jones theory of subfactors \cite{L5}, 
little is known about the possible noncommutative geometrical counterpart.

One would like to get a map
\[
\rho \longrightarrow \tau_\rho
\]
that associates a noncommutative geometric quantity $\tau_\rho$ to a sector $\rho$.

Now the operator algebraic approach to low-dimensional Conformal Quantum Field Theory (CFT) 
has shown to be very powerful as can be seen in particular by the classification of chiral 
CFTs with central charge $c<1$ \cite{KL1} and the construction of new models \cite{KL1,X6}. 
Therefore, CFT offers a natural framework for the noncommutative geometry set-up. 

Namely, we may want to look for a spectral triple associated with a sector in CFT. In order to have such a structure we may further want to restrict our attention to the supersymmetric case, namely to superconformal field theory (SCFT).

The present paper is a first step in this direction by constructing spectral triples 
associated with  (unitary, positive energy) representations of the super-Virasoro algebra 
\cite{FQS}. 

We now explain the actual content of this paper. In a recent paper by three of us 
\cite{CKL} we have set up the operator algebraic picture for SCFT. In particular, 
we have given an interpretation of Neveu-Schwarz and Ramond sectors as representations 
of a Fermi net on $S^1$ or of its promotion to the double cover of $S^1$, respectively. 

Starting with the super-Virasoro algebra, we have then defined the super-Virasoro net 
SVir$_c$ for a given admissible central charge value $c$, see \cite{FQS}. Then Neveu-Schwarz 
and Ramond representations SVir$_c$ correspond to representations of the Neveu-Schwarz 
algebra and of the Ramond algebra, respectively. As shown in \cite{CKL}, this is at least 
the case if $c< 3/2$.

In order to have the necessary tools to deal with super-derivations, we provide a quick technical summary in Section \ref{SectSuperDerivations}. Many statements are similar to the case of ungraded derivations, but specialised to our setting. In order to make clear what is meant by spectral triples and why we are interested in them, we state the classical definitions and their extensions to our setting in conformal field theory.

Our main results start in Section \ref{Sect:R} with the Ramond algebra. In this case graded representations are 
supersymmetric inasmuch as the odd element $G_0$ of the Ramond algebra is a square root of 
the shifted conformal Hamiltonian $L_0-c/24$. In the spirit of Algebraic Quantum Field 
Theory, starting from any such representation, we can define the net of von Neumann 
algebras generated by the corresponding quantum fields (the Bose and Fermi 
energy-momentum tensors). If $e^{-\beta(L_0- c/24)}$ is trace class for all 
$\beta >0$ we 
obtain a net of ($\theta$-summable) graded spectral triples by intersecting the 
local von Neumann algebras with the domain of the superderivation induced by $G_0$. 
However in principle such intersections may reduce to the multiples of the identity 
operator or in any case may be ``too small'' and this fact gives rise to a highly 
nontrivial technical problem.  

In this paper we show how to solve the above problem and in fact we prove that 
the algebra of smooth elements for the superderivation intersects every 
local von Neumann algebra in a weakly dense $^*$-subalgebra. A similar problem 
has been studied in the free supersymmetric case in \cite{BG}, where a crucial 
simplification occurred due to the Weyl commutation relations and the fact that 
the smeared free Fermi fields are bounded operators.

In particular starting from the irreducible unitary Ramond representation 
with  central charge $c$ and minimal lowest weight $h= c/24$, which is the 
unique irreducible graded unitary representation of the Ramond algebra with 
central charge $c$, we can define a nontrivial net of local even spectral 
triples.

For the Neveu-Schwarz 
algebra (in particular the vacuum sector is a representation of this algebra)
the structure is definitely less manifest because the odd 
elements $G_r$, the Fourier modes of the Fermi stress-energy tensor, 
are indexed by $r\in\mathbb Z + 1/2$, so none of them provides us with a 
supercharge operator, an odd square root of the conformal Hamiltonian. 
In fact no such Dirac type operator can exist in this case. It is however natural 
to expect that the spectral triples appearing in the Ramond case have 
a local manifestation also in the Neveu-Schwarz case.

We will indeed generalise the notion of spectral triple to the 
case where there is no supercharge operator but only a superderivation 
$\delta$ whose square $\delta^2$ is the derivation $[L_0,\,\cdot\,]$ 
implemented by the conformal Hamiltonian. The situation is here even 
different from the one treated in \cite{K} where a flow on the algebra with 
a super-KMS functional exists. 

Starting with an irreducible unitary positive energy representation of the 
Neveu-Schwarz algebra we shall construct a net of graded, generalised 
$\theta$-summable spectral triples associated with the corresponding 
Neveu-Schwarz net of von Neumann algebras. Here it is interesting to note 
that, while for the Ramond algebra we get a net of spectral triples on $S^1$, 
for the the Neveu-Schwarz algebra the net will live only on the double cover 
$S^{1(2)}$ because the local superderivations cannot be consistently 
defined on $S^1$.

For the Ramond, the JLO cocycles appear and can be investigated. 
Concerning the Neveu-Schwarz case it is unclear whether a corresponding 
cyclic cocycle can be directly defined, see Section \ref{Sect:NS}.

As we shortly mention in the outlook, we hope to continue our investigation in a 
subsequent paper where we plan to discuss related index and cohomological aspects.

\section{Preliminaries on superderivations}
\label{SectSuperDerivations}
Let $\H$ be a (complex) Hilbert space and let $\Gamma$ be a selfadjoint 
unitary operator on $\H$.
$\Gamma$ induces a $\ZZ_2$-grading $\gamma \equiv \Ad \Gamma$ on $B(\H)$. 
We shall denote $B(\H)_+$ the unital *-subalgebra of even (Bose) elements of 
$B(\H)$ and by $B(\H)_-$ the selfadjoint subspace of odd (Fermi) elements of $B(\H)$. 
Accordingly $B(\H)=B(\H)_+ \oplus B(\H)_-$. Moreover any $\gamma$-invariant 
subspace 
$L \subset B(\H)$ has a decomposition $L=L_+ \oplus L_-$, where 
$L_+ \equiv L\cap B(\H)_+$ and $L_- \equiv L\cap B(\H)_-$.

Now let $Q$ be a selfadjoint operator on $\H$  with domain $D(Q)$ and 
assume that $Q$ is odd, namely $\Gamma Q \Gamma = - Q$. 
We now define an operator ({\em superderivation}) $\delta$ on $B(\H)$ with 
domain $D(\delta)\subset B(\H)$ as follows.  

Let $D(\delta)$ be the set of operators $a \in B(\H)$ such 
that 
\begin{equation}\label{eq:def-D(delta)}
\gamma(a)Q \subset Qa -b,
\end{equation} 
for some bounded operator 
$b\in B(\H)$. Then $b$ is uniquely determined by $a$ and we set 
$\delta(a)=b$.    
Clearly $D(\delta)$ is a subspace of $B(\H)$ and the map 
$\delta :D(\delta) \mapsto B(\H)$ is linear.    
Hence we can define a norm $\|\cdot\|_1 $ on $D(\delta)$ by  
\begin{equation}
\label{norm1}
\|a \|_1 \equiv \| a\| + \| \delta(a) \|.
\end{equation}
Note also that $1 \in D(\delta)$ and $\delta(1)=0$.  

We shall now see that $D(\delta)$ is a $^*$-algebra and $\delta$ is a superderivation 
(i.e. a graded derivation). $D(\delta)$ will be called the 
\emph{domain of the superderivation $\delta = [Q,\,\cdot\,]$}. Here the brackets 
denote the super Lie-algebra brackets induced by $\Gamma$ on operators on $\H$ 
(graded commutator).  
\begin{proposition}\label{delta}
\label{delta_sProperties}
The operator $\delta$ satisfies the following properties:
\begin{itemize}
\item[$(i)$] If $a\in D(\delta)$ then $\gamma(a) \in D(\delta)$ and 
$\delta(\gamma(a))= 
-\gamma(\delta(a))$.
\item[$(ii)$] If $a\in D(\delta)$ then $a^* \in D(\delta)$ and 
$\delta(a^*)= 
\gamma(\delta(a)^*)$. 
\item[$(iii)$] If $a, b\in D(\delta)$ then $ab \in D(\delta)$ and 
$\delta(ab) =\delta(a)b+\gamma(a)\delta(b)$. 
\item[$(iv)$] $\delta$ is a weak--weak closed operator, namely 
if the net $a_\lambda \in D(\delta)$ converges to $a\in 
B(\H)$ in the weak topology and $\delta(a_\lambda)$ converges to 
$b \in B(\H)$ in the weak topology then $a \in D(\delta)$ and 
$\delta(a)=b$. 

\item[$(v)$] $D(\delta)$ is dense in $B(\H)$ in the strong topology.       

\item[$(vi)$] If $a, b \in D(\delta)$ then 
$\|\gamma(a)\|_1=\|a\|_1$, $\|a^*\|_1=\|a\|_1$ and 
$\|ab\|_1 \leq \|a\|_1\|b\|_1$, 
\end{itemize}

\end{proposition}
\begin{proof}
$(i)$ Since $\Gamma Q \Gamma =-Q$ then $\Gamma D(Q) = D(Q)$. Hence, 
if $a\in D(\delta)$ then $\gamma(a)D(Q) \subset D(Q)$ and a 
straightforward 
computation shows that, for every $\psi \in D(Q)$, 
$Q\gamma(a)\psi -aQ\psi = -\Gamma\delta(a)\Gamma \psi$. Hence,  
$\gamma(a) \in D(\delta)$ and $\delta(\gamma(a))= 
-\gamma(\delta(a))$.

\noindent $(ii)$ Let $a\in D(\delta)$ and $\psi_1, \psi_2 \in D(Q)$. 
Then, 
\begin{eqnarray*}
(a^*\psi_1, Q\psi_2) & = & (\psi_1,aQ\psi_2) = 
-(\psi_1,\delta(\gamma(a))\psi_2) + (\psi_1,Q\gamma(a)\psi_2) \\
& = & -(\delta(\gamma(a))^*\psi_1,\psi_2) +(\gamma(a^*)Q\psi_1,\psi_2).
\end{eqnarray*}
It follows that, $a^*\psi_1 \in D(Q)$ and 
$Qa^*\psi_1=\gamma(a^*)Q\psi_1 -\delta(\gamma(a))^*\psi_1.$ 
Hence, since $\psi_1 \in D(Q)$ was arbitrary, $a^* \in D(\delta)$ and 
$\delta(a^*) = -\delta(\gamma(a))^*= \gamma(\delta(a)^*)$.
\medskip

\noindent $(iii)$ Let $a,b \in D(\delta)$ and $\psi \in D(Q)$. 
Then $ab\psi, b\psi \in D(Q)$
and 
\begin{eqnarray*}
Qab\psi &=& Qab\psi - \gamma(a)Qb\psi+\gamma(a)Qb\psi- 
\gamma(a)\gamma(b)Q\psi +\gamma(ab)Q\psi \\
&=&\delta(a)b\psi+\gamma(a)\delta(b)\psi +\gamma(ab)Q\psi. 
\end{eqnarray*}
Hence $ab \in D(\delta)$ and 
$\delta(ab)=\delta(a)b+\gamma(a)\delta(b)$. 
\medskip 

\noindent $(iv)$ Let $a_\lambda \in D(\delta)$ be a net and let 
$a,b \in B(\H)$ be bounded operators such that $\lim a_\lambda= a$ and 
$\lim \delta(a_\lambda) =b$ in the weak topology of $B(\H)$ and let 
$\psi_1, \psi_2 \in D(Q)$. Then 
\begin{eqnarray*}
(a\psi_1, Q\psi_2) & = & \lim (\psi_1,a_\lambda^*Q\psi_2) =
-\lim (\psi_1,\delta(\gamma(a_\lambda^*))\psi_2) + 
\lim (\psi_1,Q\gamma(a_\lambda^*)\psi_2) \\
&=& \lim (\psi_1,\delta(a_\lambda)^*\psi_2) +
\lim (\gamma(a_\lambda)Q\psi_1,\psi_2) \\
& = & (b\psi_1,\psi_2) +(\gamma(a)Q\psi_1,\psi_2).
\end{eqnarray*}
Hence, $a\psi_1\in D(Q)$ and $Qa\psi_1=\gamma(a)Q\psi_1 +b\psi_1$ and 
since $\psi_1\in D(Q)$ was arbitrary $a\in  D(\delta)$ and 
$\delta(a)=b$.
\medskip

\noindent $(v)$ Since we have shown that  $D(\delta)$ is a unital 
*-subalgebra of $B(\H)$, by von Neumann density theorem it is enough to 
show that the commutant $ D(\delta)'$ contains only the scalar multiples of 
the identity operator. Let $t\mapsto \alpha_t$ be the 
($\sigma$-weakly) continuous one-parameter group of automorphisms of 
$B(\H)$ defined by $\alpha_t(a)=e^{itQ}ae^{-itQ}$, $a\in B(\H)$ and let 
$\tilde{\delta}$ be the corresponding generator with domain $D(\tilde{\delta})$, 
see e.g. \cite{BR1}. It is well known that $D(\tilde{\delta})$ is a strongly dense 
unital *-subalgebra of $B(\H)$ \cite{BR1}. Moreover, from the equality 
$\gamma(\alpha_t(a))=\alpha_{-t}(\gamma(a))$ it follows that $D(\tilde{\delta})$ 
is $\gamma$-invariant. Thus $D(\tilde{\delta})_+$ is strongly dense in 
$B(\H)_+=\{\Gamma\}'$ and consequently $(D(\tilde{\delta})_+)'=\{\Gamma\}''$. 
Now, it follows from \cite[Proposition 3.2.55]{BR1} that 
$D(\delta)_+=D(\tilde{\delta})_+$ and that $\delta(a)=-i\tilde{\delta}(a)$ for 
any $a\in D(\delta)_+$. Hence, $D(\delta)' \subset \{\Gamma\}''$. 
Now, if $Q=0$, $D(\delta)=B(\H)$ and there is nothing to prove. If 
$Q\neq 0$ then $Q(Q^2+1)^{-1}$ is a nonzero odd element in 
$D(\delta)$ and hence $\Gamma \notin D(\delta)'$ so that 
$D(\delta)'= \CC 1$.
\medskip

\noindent $(vi)$ The two equalities from the norm follows directly from 
$(i)$ and $(ii)$. Now let $a,b \in D(\delta)$ then, by $(iii)$ we have 
$ab\in D(\delta)$ and $\delta(ab)=\delta(a)b+\gamma(a)\delta(b)$. 
Accordingly 
$$\|ab\|_1 \leq \|a\|\|b\| + \|\delta(a)\|\|b\|+\|a \| \|\delta(b)\| 
\leq \|a\|_1\|b\|_1.$$
\end{proof}
\begin{remark} It follows from $(iv)$ of the above proposition that $\delta$ 
is $\T$ -- $\T$ closed if $\T$ is the strong, $\sigma$-weak or 
$\sigma$-strong 
topology.  Indeed, as any of such a topology $\T$ is stronger than the weak 
topology, the graph of $\delta$ is closed in the $\T$-topology of 
$B(\H)\oplus\B(\H)$ too.
\end{remark}
\begin{corollary} $D(\delta)$ with the norm $\|\cdot\|_1$ is a 
unital Banach *-algebra. 
\end{corollary}
The following lemma will be useful later. 
\begin{lemma}
\label{corelemma}
Let $D \subset D(Q)$ be a core for $Q$ let $a \in B(\H)$ and assume that 
$aD \subset D(Q)$ and the map
$D \ni \psi \mapsto Qa\psi -\gamma(a)Q\psi$
extends to a bounded linear operator $b \in B(\H)$. 
Then $a\in D(\delta)$ and $\delta(a)=b$. 
\end{lemma} 
\begin{proof} Let $\psi \in D(Q)$. By assumption $D$ is a core for $Q$ and 
thus there is a sequence $\psi_n \in D$ such that $\lim \psi_n= \psi$ 
and $\lim Q\psi_n=Q\psi$. Hence, $\lim a\psi_n= a\psi$ and 
$\lim Qa\psi_n =\gamma(a)Q\psi +b\psi$ and since $Q$, being selfadjoint,
is a closed operator, $a\psi \in D(Q)$ and $Qa\psi = 
\gamma(a)Q\psi+b\psi$. Since $\psi \in D(Q)$ was arbitrary we have proved 
that $\gamma(a)Q \subset Qa -b$ and the conclusion follows. 
\end{proof}
We now consider the domains $D(\delta^n)$, $n \in \NN$, of the powers of 
$\delta$. 
Note that $D(\delta)=D(\delta^1)\supset D(\delta^2) \supset 
D(\delta^3) \dots$ and that $C^\infty(\delta) \equiv 
\cap_{n=1}^\infty D(\delta^n)$ is $\delta$-invariant. 

We define a norm $\|\cdot \|_n$ on $D(\delta^n)$, $n\in \NN$ by \eqref{norm1} and the
recursive relation 
\begin{equation}
\label{normn}
\|a\|_{n+1}= \|a\|_n + \|\delta(a)\|_n.
\end{equation}      

\begin{proposition}
\label{proposition*subalgebras}
The subspaces $D(\delta^n)$, $n \in \NN$ and $C^\infty(\delta)$ are
$\gamma$-invariant unital *-subalgebras of $B(\H)$. Moreover the pair 
$\big(D(\delta^n),\|\cdot \|_n\big)$ is a Banach *-algebra
for all $n\in \NN$ and $ \|\gamma(a)\|_n  = \|a\|_n$ for 
all $a \in D(\delta^n)$. 
\end{proposition}
\begin{proof} 
Clearly it is enough to prove the proposition for the subspaces 
$D(\delta^n)$, $n \in \NN$. We proceed by induction.

From Proposition \ref{delta_sProperties} $D(\delta^1)=D(\delta)$, 
is a $\gamma$-invariant unital *-subalgebra of $B(\H)$. Moreover, 
it is a Banach *-algebra with the norm $\|\cdot \|_1$. 
Assume now that the same is true for the pair 
$\big( D(\delta^n), \|\cdot \|_n\big)$. Since $1 
\in D(\delta)$ and 
$\delta(1)=0$ we can conclude that $1\in D(\delta^{n+1})$. Now let 
$a,b \in D(\delta^{n+1})$. Then, $\gamma(a), a^*, ab \in D(\delta)$ 
and $\delta(a), \delta(b) \in D(\delta^n)$. Moreover 
$\delta(\gamma(a)) = -\gamma(\delta(a)) \in D(\delta^n)$, 
$\delta(a^*) = \gamma(\delta(a)^*) \in D(\delta^n)$ and 
$\delta(ab) = \delta(a)b+\gamma(a)\delta(b) \in D(\delta^n)$. 
Hence, $D(\delta^{n+1})$ is a unital $\gamma$-invariant *-subalgebra 
of $B(\H)$. That the norm $\|\cdot \|_{n+1}$ is a *-algebra norm 
$D(\delta^{n+1})$ and that it is $\gamma$-invariant follows exactly as 
in the proof of Proposition \ref{delta_sProperties} $(vi)$ and it remains 
to show that $D(\delta^{n+1})$ is complete. Let $a_m$, $m\in \NN$ be a 
Cauchy sequence in $D(\delta^{n+1})$. By the inductive assumption 
$a_m$ and $\delta(a_m)$ converge to elements $a, b \in 
D(\delta^{n})$ respectively and it follows from 
Proposition \ref{delta_sProperties} $(iv)$ that $b=\delta(a)$ and hence 
$a \in D(\delta^{n+1})$ and $\|a_m - a\|_{n+1}$ tends to $0$ as $m$ 
tends to $\infty$.  
\end{proof} 

For every $a \in B(\H)$ we denote by $\sigma(a)$ the spectrum of $a$. 
The following proposition can be proved adapting the proof of 
\cite[Proposition 3.2.29]{BR1}. 
\begin{proposition}
If $a \in D(\delta)$ and $\lambda \notin \sigma(a)$ then 
$(a-\lambda 1)^{-1} \in D(\delta)$ and 
\begin{equation}
\delta\big( (a-\lambda 1)^{-1}\big)=
-(\gamma(a)-\lambda 1)^{-1} \delta(a)(a-\lambda 1)^{-1}.
\end{equation}
\end{proposition}

\begin{corollary} For all $n\in \NN$, if $a \in D(\delta^{n})$ 
and $\lambda \notin \sigma(a)$ then $(a-\lambda 1)^{-1} \in 
D(\delta^{n})$. 
\end{corollary}

\begin{corollary} For all $n\in \NN$, if $a \in D(\delta^{n})$
and $f$ is a complex function holomorphic in a neighbourhood of $\sigma(a)$ 
then $f(a) \in D(\delta^{n})$.
\end{corollary} 
Now consider the positive selfadjoint operator $H \equiv Q^2$ and the 
corresponding derivation $\delta_0$ on $B(\H)$. Then the generator
of the one-parameter group of automorphisms  ${\rm Ad}e^{itH}$ of 
$B(\H)$ is $i\delta_0$. Note that $H$ commutes with $\Gamma$. 

\begin{lemma}\label{lemma:deltaSquare} 
If $a\in D(\delta^2)$ then $a \in D(\delta_0)$ 
and $\delta^2(a)=\delta_0(a)$. 
\end{lemma}
\begin{proof} Assume that $a\in D(\delta^2)$ and that 
$\psi \in D(H) \subset D(Q)$. Then $a\psi \in D(Q)$ and 
$Qa\psi -\gamma(a)Q\psi = \delta(a)\psi$. Now, 
$Q\psi \in D(Q)$ and moreover, since $\gamma(a) \in D(\delta)$
and $\delta(a)\in D(\delta)$ we have 
$\delta(a)\psi+\gamma(a)Q\psi \in D(Q)$. Hence $a\psi \in D(H)$
and 
\begin{eqnarray*}
Ha\psi & = & Q^2a\psi = Q\delta(a)\psi+Q\gamma(a)Q\psi \\
& = & \delta^2(a)\psi +\gamma(\delta(a))Q\psi + 
\delta(\gamma(a))Q\psi + aH\psi \\
& = & \delta^2(a)\psi + aH \psi .
\end{eqnarray*}
Since $\psi \in D(H)$ was arbitrary, the conclusion follows from 
\cite[Proposition 3.2.55]{BR1}.
\end{proof}
For any $f \in L^1(\RR)$ on $\RR$  and any $a \in B(\H)$ we define
\begin{equation} 
a_f \equiv \int_\RR e^{itH}ae^{-itH}f(t){\rm d}t .  
\end{equation} 
\begin{lemma}
\label{lemmaA_f} 
Assume that $a \in D(\delta)$ and that 
$f\in L^1(\RR)$. Then $a_f \in D(\delta)$ and 
$\delta(a_f)=\delta(a)_f$.
\end{lemma}
\begin{proof} This is a straightforward consequence of the fact 
that the one-parameter group of unitaries $e^{itH}$ commutes
with $Q$ and $\Gamma$. 
\end{proof}
\begin{lemma} 
\label{lemmaA_f2}
Assume that $a \in D(\delta)$ and that
$f\in C^\infty_c(\RR)$. Then $a_f \in D(\delta^2)$ and 
$\delta^2(a_f)=ia_{f'}$.
\end{lemma}
\begin{proof} For any $\psi \in D(H)$, we have $Q\psi \in D(Q)$.
Moreover, a standard and straightforward argument shows that 
the map $t\mapsto e^{iHt}a_fe^{-iHt} \in B(\H)$ is differentiable at $t=0$  
and that the corresponding derivative is equal to $-a_{f'}$. 
Hence, by \cite[Proposition 3.2.55]{BR1}, 
$a_f\psi \in D(H)$ and $Ha_f\psi= ia_{f'}\psi +a_fH\psi$.  
Similarly $\delta (a)_f \psi \in D(H)$. It follows that  
$\delta(a_f)\psi=\delta(a)_f\psi \in D(Q)$ and 
\begin{eqnarray*}
Q\delta(a_f)\psi & = & Ha_f\psi-Q\gamma(a)_fQ\psi  \\
& = & Ha_f\psi -\delta(\gamma(a_f))Q\psi -a_fH\psi \\
& = & \delta_0(a_f)\psi-\delta(\gamma(a_f))Q\psi \\
& = & ia_{f'}\psi -\delta(\gamma(a_f))Q\psi \\
& = & ia_{f'}\psi + \gamma(\delta(a_f))Q\psi. 
\end{eqnarray*} 
Since $D(H)$ is a core for $Q$ it follows by Lemma \ref{corelemma}
that $a_f \in D(\delta^2)$ and $\delta^2(a_f)=ia_{f'}$.
\end{proof}
From Lemma \ref{lemmaA_f} and Lemma \ref{lemmaA_f2} the following
proposition can be easily proved by induction. 
\begin{proposition}
\label{propositionA_f}
Assume that $a \in D(\delta)$ and that 
$f\in C^\infty_c(\RR)$. Then $a_f\in C^\infty(\delta)$. 
\end{proposition}   
\begin{corollary} $C^\infty(\delta)$  is a core for $\delta$ 
(with respect to the $\sigma$-weak topology), namely 
$\delta$ coincides with the ($\sigma$-weak)--($\sigma$-weak) closure of 
its restriction to $C^\infty(\delta)$. 
\end{corollary}
\begin{corollary} $C^\infty(\delta)$ is dense in $B(\H)$ in the 
strong topology.
\end{corollary}

\section{Spectral triples in conformal field theory}
\label{SectSpectralTriples}
The purpose of this section is to state our definitions of spectral triple and 
give a few comments on Connes definition and related matters.

Firstly we will state the definitions suitable for the Ramond algebra case.
\begin{definition}\label{def:ST1}
A ($\theta$-summable)  {\em graded spectral triple} $(\gA, \H, Q)$ consists of a 
graded Hilbert space $\H$, where the selfadjoint grading unitary is denoted by 
$\Gamma$, a unital $^*$-algebra $\gA \subset B(\H)$ graded by $\gamma \equiv \Ad 
(\Gamma) $, and an odd selfadjoint operator $Q$ on $\H$ as follows:
\begin{itemize}
\item $\gA$ is contained in $D(\delta)$, the domain of the superderivation 
$\delta =[Q,\,\cdot\,]$  as in Sect. \ref{SectSuperDerivations};
\item For every $\beta > 0$, Tr$(e ^{-\beta Q^2})<\infty$ 
($\theta$-summability).
\end{itemize}
The operator $Q$ is called the {\em supercharge} operator, its square the 
{\em Hamiltonian}.
\end{definition}
\begin{remark}\label{Connesdefinition}
Restricting to the even subalgebra $\gA_+$ of $\gA$, the above definition is 
essentially Connes  \cite{C} (see also \cite{CM}) definition of a (even) 
spectral triple $(\gA_+, \H, Q)$. 
This is the fundamental object for index theorems and evaluating on 
$K$-theory elements. In this case the supercharge $Q$ 
is traditionally called {\em Dirac operator} and denoted by $D$. 
\end{remark}
\begin{remark}\label{remarkExampleSTriple} 
Let $\H$ Hilbert space graded by $\Gamma$ and let  
$\A$ be a unital *-subalgebra of $B(\H)$ such that 
$\gamma(\A)=\A$. Let moreover $\gA \equiv \A\cap 
C^\infty(\delta)$. 
Then, provided Tr$(e ^{-\beta Q^2})<\infty$ for all $\beta > 0$, 
$(\gA , \H, Q)$ is a graded spectral triple  in the sense of 
Definition \ref{def:ST1}. Note also that if $\A$ is a von Neumann algebra then 
$\delta$ restricts to a weak-weak closed superderivation of $\A$.
\end{remark}
Our spectral triples will satisfy an additional property which is described 
in the following definition taken from \cite{JLO}.
\begin{definition}\label{def:Qalgebra}
A {\em quantum algebra}  $(\gA, \H, Q)$
is a ($\theta$-summable) graded spectral triple such that 
$\delta (\gA) \subset \gA$. 
\end{definition}
\begin{remark}\label{orderone}
Let $(\gA, \H, Q)$ be a quantum algebra, thus the additional property 
$\delta(\gA) \subset \gA$ is satisfied, and let $J:\H \mapsto \H$
be an antiunitary involution such that $J\gA J \subset \gA'$.
Then we have $J\gA J \subset \delta(\gA)'$
which is essentially the order one condition for the operator $Q$ 
(see e.g. \cite{CM}).
If we restrict to the associated even spectral triples we have
$\delta(\gA_+)\subset \gA_-$ and hence $\delta(\gA_+)\cap \gA_+ =\{ 0 \}$.
However $J\gA_+J \subset \delta(\gA)' \subset \delta(\gA_+)'$ and the
order one condition for $Q$ is still satisfied.
\end{remark}
\begin{remark}\label{remarkQalgebra} 
If $(\gA, \H, Q)$ is a quantum algebra clearly we have 
$\gA \subset C^\infty(\delta)$. Conversely let
$\H$ and $\A$ as in Remark \ref{remarkExampleSTriple}. 
Suppose that $\delta(a) \in \A$ for every $a \in \A \cap D(\delta)$ and let $\gA 
\equiv \A\cap C^\infty(\delta)$. Then, provided Tr$(e ^{-\beta Q^2})<\infty$ 
for all $\beta > 0$, $(\gA , \H, Q)$ is a quantum algebra.
\end{remark}
\begin{remark}\label{rem:ST}
The supercharge operator $Q$ appears in supersymmetric field theories and its 
square $Q^2$ is the Hamiltonian. In conformal field theory, the subject of 
this paper, it will be (up to an additive constant) the conformal Hamiltonian 
$L^\l_0$ in the considered representation $\lambda$, c.f. Section 
\ref{Sect:R}. 
Then the $\theta$-summability condition is automatically satisfied under very 
general conditions.
\end{remark}
In the Neveu-Schwarz case, see Section \ref{Sect:NS}, we will have a 
Hamiltonian $H$ and a superderivation $\delta$ on the algebra $\gA$, 
{\em without} a supercharge operator. Namely there can be no odd selfadjoint 
operator $Q$ satisfying $Q^2 =H$ and $\delta = [Q,\,\cdot\,]$. To treat also 
this 
case we need to generalise the definition of spectral triple. However to express 
the condition $Q^2=H$ in terms of the superderivation $\delta$ we need to give a 
meaning to its square $\delta^2$. We are thus led to assume from the beginning
the additional condition $\delta(\gA) \subset \gA$ and thus to generalise only the 
notion of quantum algebra. This will suffice for the purposes of this paper. 
\begin{definition}\label{def:Qalgebra2}
 A {\em generalised quantum algebra} $(\gA , \H, \delta)$ 
consists of a graded Hilbert space $\H$, where the selfadjoint grading unitary is 
denoted by $\Gamma$, a unital *-algebra $\gA \subset B(\H)$ graded by $\gamma 
\equiv \Ad (\Gamma) $, and an antisymmetric odd superderivation 
$\delta: \gA \rightarrow \gA$, i.e., a linear map satisfying
\begin{eqnarray*}
\delta(a^*) &=& -\delta(\gamma(a))^* \\
\delta(\gamma(a)) &=& -\gamma(\delta(a)) \\ 
\delta(ab) &=& \delta(a)b + \gamma(a)\delta (b)
\end{eqnarray*}
$a,b \in D(\delta)$, with the following properties:
\begin{itemize}
\item $\delta$ is $\sigma$-weakly closable, i.e. it extends to a 
($\sigma$-weakly)--($\sigma$-weakly) closed superderivation of the 
von Neumann algebra $\gA''$. 

\item There exists an even positive selfadjoint operator $H$ on 
$\H$ (the Hamiltonian) such that
for every  $a \in \gA$ and every $\psi \in D(H)$, $a\psi\in D(H)$ and 
$Ha\psi- aH\psi= \delta^2(a)\psi.$
 \item[$\bullet$] For every $\beta > 0$, the operator $e ^{-\beta H}$ is of 
trace class.
\end{itemize}
\end{definition}
In the following two sections we will construct spectral triples of the above  
types. Indeed we shall have nets of spectral triples in the following sense.
Let $\I$ be the family of nonempty, nondense, open intervals of
the unit circle $S^1=\{z\in \CC: |z|=1\}$ and let 
\[
\I_0 \equiv \{I\in \I: \overline{I} \subset S^1\setminus\{-1\}\}, 
\]
where $\overline{I}$ denotes the closure of the interval $I\in \I$. 
\begin{definition}
A {\em net of graded spectral triples} $(\gA , \H, Q)$ on $S^1$ 
(resp. $S^1\setminus\{-1\}$) 
consists of graded Hilbert space $\H$, an odd selfadjoint operator $Q$ and 
a net $\gA$ of unital *-algebras on $\I$ (resp. $\I_0$) acting on $\H$, 
i.e. a map from $\I$ (resp. $\I_0$) into the family of unital *-subalgebras of
$B(\H)$ which satisfies isotony property 
$$\gA(I_1)\subset \gA(I_2)\quad {\rm if}\; I_1\subset I_2,$$
such that $(\gA(I) , \H, Q)$ is a graded spectral triple for all $I\in \I$
(resp. $I\in \I_0$). 
If the net satisfies the additional property 
$\delta(\gA(I))\subset \gA(I)$, $\delta =[Q,\,\cdot\,]$, for all $I\in \I$ 
(resp. $I\in \I_0$) then we say that $(\gA , \H, Q)$ is a {\em net of quantum algebras}
 on $S^1$ (resp. $S^1\setminus\{-1\}$). 
\end{definition}
Now we give a more general definition to cover the case where there is no global 
supercharge operator. In this context the nets will be on the double cover of 
$S^1$, \cite[Sect. 3.2]{CKL}. 
Denote by $\I^{(n)}$ the intervals on the $n$-cover $\Sn$ of $S^1$, namely 
$I\in\I^{(n)}$ if $I$ is a connected subset of $\Sn$ whose projection onto the base 
$S^1$ belongs to $\I$.
\begin{definition}\label{def:netQalgebra2}
A {\em net of generalised quantum algebras} $(\gA , \H, \delta)$ 
on $\Sn$ (resp. $S^1\setminus\{-1\}$) consists of a graded Hilbert space $\H$, of a 
net of unital *-algebras on $\Sn$ (resp. $S^1\setminus\{-1\}$) acting on $\H$ and a 
net $\delta$ of superderivations on 
$\gA$, i.e. a map $I\in \I \mapsto \delta_I$ (resp. $I\in \I_0 \mapsto 
\delta_I$), where $\delta_I:\gA(I) \mapsto \gA(I)$ is a superderivation, 
satisfying $\delta_{I_2}|_{\gA(I_1)} = \delta_{I_1}$ if
$I_1\subset I_2$, such that $(\gA(I) , \H, \delta_I)$ is a generalised quantum 
algebra for every $I\in \I$ (resp. $I \in \I_0$) with Hamiltonian $H$ independent 
of $I$. 
\end{definition}
Note that a net on $\Sn$ gives rise to a net on $S^1\setminus\{-1\}$ by restriction. Conversely, if rotation covariance holds true, a net on $S^1\setminus\{-1\}$ extends to a net on $\Sn$ for some finite or infinite $n$.

\section{Spectral triples from the Ramond algebra}\label{Sect:R}
In this section we shall construct nets of quantum algebras from 
representations of the Ramond (Super-Virasoro) algebra. 

Recall that the Ramond algebra is the super-Lie algebras generated by even 
elements $L_n$, $n\in\mathbb Z$, odd elements $G_r$, $r \in \ZZ$, and a central 
even element $k$ satisfying the relations
\begin{equation}
\begin{gathered}\label{svirdef}
    [L_m , L_n] = (m-n)L_{m+n} + \frac{k}{12}(m^3 - m)\de_{m+n, 0},\\
    [L_m, G_r] = (\frac{m}{2} - r)G_{m+r},\\
    [G_r, G_s] = 2L_{r+s} + \frac{k}{3}(r^2 - \frac14)\de_{r+s,0}.
    \end{gathered}
\end{equation}
We shall consider representations $\lambda$ of the Ramond algebra by linear 
endomorphisms, denoted by $L^\l_m, G^\l_r, k^\l$, $m, r \in \ZZ$, of a complex 
vector 
space $V_\lambda$ equipped with an involutive linear endomorphism $\Gamma_\l$ 
inducing the 
super-Lie algebra grading. The endomorphisms $L^\l_m, G^\l_r, k^\l$, satisfies 
the relations \ref{svirdef} with respect to the brackets given by the super-commutator 
induced by $\Gamma_\l$. 

We assume that the representation $\l$ satisfies the following properties. 
\medskip

\noindent $(i)$ $k^\l = c 1$ for some $c\in \CC$ (the {\em central charge} of 
the representation $\l$). 

\noindent $(ii)$ $L^\l_0$ is diagonalizable on $V_\l$, namely 
\begin{equation}
V_\l =\bigoplus_{\alpha \in \CC} {\rm Ker}(L^\l_0 - \alpha 1)
\end{equation}

\noindent $(iii)$ ${\rm Ker}(L^\l_0 - \alpha 1)$ is finite-dimensional for all 
$\alpha \in \CC$. 
\medskip

\noindent $(iv)$ ({\em Unitarity}) There is a scalar product $(\cdot,\cdot)$ 
on $V_\l$ such that 
\begin{eqnarray*}
(L^\l_m u, v) = (u,L^\l_{-m}v), \\
(G^\l_r u, v) = (u,G^\l_{-r}v), \\
(\Gamma_\l u, v) = (u,\Gamma_\l v),
\end{eqnarray*}
for all $u, v \in V_\l$ and all $m, r \in \ZZ$. 
\medskip 

As a consequence of the above assumptions the central charge $c$ is
a real number and 
\[
{\rm Ker}(L^\l_0 - \alpha 1)= \{0\}
\]
if $\alpha$ is not a 
real number such that $\alpha \geq c/24$. Since the eigenvalues of 
$L^\l_0$ are real numbers bounded from below it follows from the (unitary) 
representation theory of the Virasoro algebra that $c\geq 0$ and hence $\l$ is 
a {\em positive energy representation}.
Accordingly,  the possible values of the central charge are either 
$c\geq 3/2$  or
\begin{equation}\label{c-values}
c = \frac32 \left( 1 - \frac{8}{m(m+2)}\right), \ m=2,3,\ldots
\end{equation}
see \cite{FQS}. 

Note that the graded unitary lowest weight representations of the Ramond algebra 
satisfy all the above assumptions.   

Now let $\l$ be a representation of the Ramond algebra satisfying assumptions 
$(i)$--$(iv)$ and let $\H_\l$ be the Hilbert space completion of $V_\l$. 
The endomorphisms $L_m^\l, G_r^\l$, $m, r \in \ZZ$, define unbounded operators
on $\H_\l$ with domain $V_\l$, which are closable since by assumption $(iv)$
they have densely defined adjoint. We shall denote their closure 
by the same symbols. With this convention $L^\l_0$ is a selfadjoint 
operator on $\H_\l$. Moreover $\Gamma_\l$ extends to a selfadjoint unitary 
operator on $\H_\l$ which will also be denoted by the same symbol.  

The operators $L^\l_m, G^\l_r$,
$m, r \in {\mathbb Z}$ satisfy the {\it energy bounds}  
\begin{equation}
\label{e-boundsL}
\|L^\l_m v\|\leq M (1+|m|^{\frac{3}{2}})\|(1+L^\l_0)v\|, \quad v\in V_\l, 
\end{equation}
for a suitable constant $M>0$ depending on the central charge $c$ and 

\begin{equation}
\label{e-boundsG}
\|G^\l_r v\|\leq (2+ \frac{c}{3}r^2)^{\frac12}\|(1+L^\l_0)^{\frac12}v\|, \quad 
v\in V_\l,
\end{equation}
see \cite[Sect. 6.3]{CKL} and the references therein.

Now let $f$ be a smooth function on $S^1$. It follows from the linear 
energy bounds in Equations (\ref{e-boundsL}), (\ref{e-boundsG}) and the fact 
that the Fourier 
coefficients
\begin{equation}
\hat{f}_n=\int_{-\pi}^\pi f(e^{i\theta})e^{-in\theta}\frac{{\rm d}\theta}{2\pi},\;
n\in \ZZ,
\end{equation}
are rapidly decreasing, that the maps 
\begin{equation}
V_\l \ni v \mapsto  \sum_{n \in \ZZ}\hat{f}_nL^\l_nv 
\end{equation}

\begin{equation}
V_\l \ni v \mapsto \sum_{r \in \ZZ }\hat{f}_rG^\l_r v
\end{equation}
define closable operators on $\H_\l$ and we shall denote by 
$L^\l(f)$ and $G^\l(f)$ ({\it smeared fields}) respectively 
the corresponding closures. 
Their domains contain $D(L^\l_0)$ and they leave invariant 
$C^\infty(L^\l_0)$. Moreover, if $f$ is real, $L^\l(f)$ and $G^\l(f)$ are 
selfadjoint operators which are essentially selfadjoint on any core for 
$L^\l_0$, cf. \cite{BS-M}. 

Using these smeared fields we shall define a net of von Neumann algebras in 
the usual way. Let $\I$ as in Section \ref{SectSpectralTriples}. We define
a net ${\A_\l}$ of von Neumann algebras on $S^1$ by 
\begin{equation}
\label{gener1} 
\A_\l(I) \equiv \{ e^{iL^\l(f)}, e^{iG^\l(f)}:f\in C^{\infty}(S^1) 
\;{\rm real},\, {\rm supp}f \subset I\}'', \; I \in \I. 
\end{equation} 
It is clear from the definition that isotony is satisfied, namely
\begin{equation}
\label{isotony}
\A_\l(I_1) \subset \A_\l(I_2) \quad {\rm if}\; I_1 \subset I_2. 
\end{equation}
In the same way, we see that $L^\l(f)$ and $G^\l(f)$ are affiliated with $\A_\l(I)$ if ${\rm supp}f \subset I$. Moreover each algebra $\A_\l(I)$, $I\in \I$, is left globally invariant
by the grading automorphism $\gamma_\l \equiv {\rm Ad} \Gamma_\l$ of 
$B(\H_\l)$. 

Let $\Diff(S^1)$ be the group of (smooth) orientation preserving 
diffeomorphisms of $S^1$ and let $\Diff^{(\infty)}(S^1)$ be its universal 
cover. Moreover let $\Diff_I(S^1) \subset \Diff(S^1)$ be 
the subgroup of diffeomorphisms that are localised in $I$, namely
that act trivially on $I'$, and let $\Diff_I^{(\infty)}(S^1)$  be the 
connected component of the identity of the pre-image of 
$\Diff_I(S^1)$ in $\Diff^{(\infty)}(S^1)$. 
We denote by $u_\l: \Diff^\infty(S^1) \mapsto U(\H_\l)/{\mathbb T}$ the strongly 
continuous projective unitary positive energy representation obtained by 
integrating the restriction of the representation $\l$ to the Virasoro Lie 
subalgebra of the Ramond algebra, see \cite{GoWa,Tol99} and for every 
$g \in \Diff^\infty(S^1)$ we choose a unitary operator $U_\l(g)$ 
in the equivalence class $u_\l(g) \in U(\H_\l)/{\mathbb T}$. 
If $g\in \Diff_I^{(\infty)}(S^1)$ one can show that 
$U_\l(g) \in \A_\l (I)$, see the proof of \cite[Theorem 33]{CKL}. 
Note that if $\l$ is a lowest weight representation with lowest 
weight $h_\l$ then $e^{i2\pi L^\l_0}=e^{i2\pi h_\l}$ and the projective 
unitary representation $u_\l$ factors through $\Diff(S^1)$.  

Arguing as in \cite[Sect. 6.3]{CKL} where a similar
construction has been carried out for the vacuum representations of the
Neveu-Schwarz (super-Virasoro) algebra we obtain the following theorem.  
We shall omit the details of the proof which in the case of the Ramond algebra 
are similar but in fact simpler.

\begin{theorem} The net $\A_\l$ satisfies the following properties: 
\begin{itemize}
\item[$(i)$] (Graded locality) If $I_1, I_2 \in \I$ and 
$I_1 \cap I_2 = \emptyset$ then 
$ \A_\l(I_1) \subset Z_\l\A_\l(I_2)'Z_\l^*,$
where $Z_\l \equiv (1-i\Gamma_\l)/(1-i)$. 
\item[$(ii)$] (Conformal covariance) 
$U_\l(g)\A_\l(I)U_\l(g)^*= \A_\l(\dot{g}I),$
for all $g\in \Diff^{(\infty)}(S^1)$ and all $I\in \I$,
where $\dot{g}$ denotes the image of $g$ in $\Diff(S^1)$ under the covering 
map. 
\end{itemize}
\end{theorem}

Note that for any $\l$ (not necessarily a lowest weight representation)
we have 
\begin{equation} 
e^{i2\pi L^\l_0} \in \bigcap_{I\in \I}\A_\l(I)'
\end{equation}
and hence the action of $\Diff^{(\infty)}(S^1)$ on the local algebras 
factors through $\Diff(S^1)$.

We now define a supercharge operator $Q_\l$ on $\H_\l$ by 
$Q_\l \equiv G^\l_0$. $Q_\l$ is an odd operator, namely 
$\Gamma_\l Q_\l\Gamma_\l=-Q_\l$. Moreover it satisfies 
\begin{equation}
Q_\l^2 = H_\l \equiv L^\l_0-\frac{c}{24}. 
\end{equation}

We denote by $\delta$ the corresponding superderivation as defined in Section 
\ref{SectSuperDerivations}. We now define a net of unital *-subalgebras on 
$S^1$ by $\gA_\l(I) \equiv \A_\l(I) \cap C^\infty (\delta)$, $I \in \I$.
Our aim is to show that this net  satisfies $\delta(\gA_\l(I))\subset \gA_\l(I)$ 
and that it is strong operator dense in $\A_\l(I)$ for any $I\in \I$. 
This will give a net of quantum algebras on $S^1$ naturally associated with the 
net of von Neumann algebras $\A_\l$ and hence a noncommutative geometric structure 
\cite{C} associated to superconformal quantum field theories. 
Inspired by the work of Buchholz and 
Grundling \cite{BG} we shall use resolvents of the smeared Bose fields to exhibit 
local elements in the domain of $\delta$. Yet the models considered 
here appear to be more complicated than the free field model considered by them. 

According with the notation in Sect. \ref{SectSuperDerivations} we set, for all 
$a\in B(\H_\l)$ and $f\in C^\infty_c(\RR)$, 
\begin{equation}
a_f \equiv \int_\RR e^{itH_\l}ae^{-itH_\l}f(t){\rm d}t =
\int_\RR e^{itL^\l_0}ae^{-itL^\l_0}f(t){\rm d}t .
\end{equation}

\begin{proposition}\label{prop:domain-invariance}
$\delta(\gA_\l(I))\subset \gA_\l(I)$ for all 
$I\in \I$.
\end{proposition}
\begin{proof} Fix an arbitrary interval $I\in \I$ and an arbitrary $a\in 
\A_\l(I)\cap D(\delta)$. 
As a consequence of Remark \ref{remarkQalgebra} it is enough to show that 
$\delta(a) \in \A_\l(I)$. Let $I_1, I_2 \in \I$ be such that the
closure $\overline{I}$ of $I$ is contained in $I_1$ and 
$\overline{I_1} \subset I_2$. 
Then, if the support of $f\in C^\infty_c(\RR)$ is sufficiently close to $0$,
$a_f \in \A_\l(I_1).$ By Lemma \ref{lemmaA_f}, $a_f\in D(\delta)$
and $\delta(a_f) =\delta(a)_f$. Moreover, a standard argument shows that 
$a_f C^\infty(L^\l_0) \subset C^\infty(L^\l_0)$. 

Now let $\varphi_1$ and $\varphi_2$ be two real nonnegative smooth functions on 
$S^1$ such that ${\rm supp}\varphi_1 \subset I_2$, 
${\rm supp}\varphi_2 \subset I_1'$ and $\varphi_1+\varphi_2=1$ and let 
$\psi \in C^\infty(L^\l_0)$. 
Then, 
\begin{eqnarray*}
\delta(a)_f\psi  &=& \delta(a_f)\psi = 
Q_\l a_f \psi -\gamma_\l(a_f)Q_\l \psi \\
&=& G^\l(\varphi_1)a_f\psi + G^\l(\varphi_2)a_f\psi - 
\gamma_\l(a_f)G^\l(\varphi_1)\psi - \gamma_\l(a_f)G^\l(\varphi_2)\psi.
\end{eqnarray*}
Since $G^\l(\varphi_2)$ is affiliated with 
$\A_\l(I_1') \subset Z_\l\A_\l(I_1)'Z_\l^*$ (using graded locality), we have 
$$G^\l(\varphi_2)a_f\psi -\gamma_\l(a_f)G^\l(\varphi_2)\psi=0.$$ 
Hence, 
$$\delta(a)_f\psi=  G^\l(\varphi_1)a_f\psi -
\gamma_\l(a_f)G^\l(\varphi_1)\psi.$$  
Then, given an arbitrary $b\in \A_\l(I_2)'$, we have
\begin{eqnarray*} 
b\delta(a)_f\psi &=&  bG^\l(\varphi_1)a_f\psi -
b\gamma_\l(a_f)G^\l(\varphi_1)\psi \\
&=& 
G^\l(\varphi_1)a_fb\psi -\gamma_\l(a_f)G^\l(\varphi_1)b\psi.
\end{eqnarray*}
Since $C^\infty(L^\l_0)$ is a core for $G^\l(\varphi_1)$ we can find 
a sequence $\psi_n \in C^\infty(L^\l_0)$ such that 
$\psi_n$ tends to $b\psi$ and $G^\l(\varphi_1)\psi_n$ tends to 
$G^\l(\varphi_1)b\psi$ as $n$ tends to $\infty$. 
Then   
\begin{eqnarray*}
\lim_{n\to \infty} G^\l(\varphi_1)a_f\psi_n 
&=& \lim_{n\to \infty} \left(\delta(a)_f\psi_n + 
\gamma_\l(a_f)G^\l(\varphi_1)\psi_n \right) \\
&=& \delta(a)_fb\psi + \gamma_\l(a_f)G^\l(\varphi_1)b\psi
\end{eqnarray*} 
and since $G^\l(\varphi_1)$ is a closed operator it follows that  
$$G^\l(\varphi_1)a_fb\psi= 
\delta(a)_fb\psi + \gamma_\l(a_f)G^\l(\varphi_1)b\psi.$$ 
Hence
\begin{eqnarray*}
b\delta(a)_f\psi &=& G^\l(\varphi_1)a_fb\psi 
-\gamma_\l(a_f)G^\l(\varphi_1)b\psi \\
&=& \delta(a)_fb\psi + \gamma_\l(a_f)G^\l(\varphi_1)b\psi
-\lim_{n \to \infty}\gamma_\l(a_f)G^\l(\varphi_1)\psi_n  \\
&=&\delta(a)_fb\psi.
\end{eqnarray*}
It follows that $\delta(a)_f \in \A_\l(I_2)$ for every smooth function 
$f$ on $\RR$ with support sufficiently close to $0$. Hence, 
$\delta(a) \in \A_\l(I_2)$ and since $I_2$ can be any interval in $\I$ 
containing the closure of $I$, 
$$\delta(a) \in \bigcap_{I_0 \supset \overline{I}} \A_\l(I_0).$$
The conclusion follows since the latter intersection of von Neumann 
algebras coincides with $\A_\l(I)$ as a consequence of conformal 
covariance. \end{proof}

To show that ${\gA}_\l(I)$ is strong operator dense in 
${\A}_\l(I)$ for all $I \in \I$ we need some preliminary results. 

\begin{proposition}
\label{resolvent}
For every $k \in \NN$ and every real 
$f \in C^\infty(S^1)$ there exists a real number $M>0$ such that, for 
every $\alpha \in \RR$ satisfying $|\alpha|>M$ the following holds
$$(L^\l(f)+i\alpha)^{-1}D((L^\l_0)^k) \subset D((L^\l_0)^k).$$
\end{proposition}
\begin{proof} Let $(\cdot,\cdot)_k$ be the scalar product 
on $D((L^\l_0)^k)$ given by 
$$(\psi_1,\psi_2)_k \equiv ((L^\l_0+1)^k\psi_1,(L^\l_0+1)^k\psi_2).$$ 
With this scalar product $D((L^\l_0)^k)$ is a Hilbert space which we shall 
denote by $\H^k$. Let $\|\cdot \|_k$ be the corresponding norm. 
By \cite[Proposition 2.1]{Tol99}, $e^{itL^\l(f)}$, $t\in \RR$, restricts 
to bounded linear maps $\H^k \to \H^k$ satisfying
$$\|e^{itL^\l(f)}\|_{B(\H^k)}\leq e^{|t|M},$$
for suitable constant $M>0$ (depending on $f$ and $k$). 
Moreover, it follows from \cite[Corollary 2.3]{Tol99}
(see also \cite[Lemma 3.1.1]{Tol99}) that the map 
$t \to e^{itL^\l(f)} \in B(\H^k)$ is strongly continuous. 

Now let $\alpha >M$ and let $\psi \in D((L^\l_0)^k)$ then
on $\H_\l$ we have the equality 
$$(L^\l(f)+i\alpha)^{-1}\psi = 
-i\int_0^{\infty}e^{itL^\l(f)}e^{-t\alpha}\psi {\rm d}t.$$ 
The map $t\to e^{itL^\l(f)}e^{-t\alpha}\psi \in \H^k$ is 
continuous and 
\begin{equation*}
\int_0^{\infty}\|e^{itL^\l(f)}e^{-t\alpha}\psi\|_k {\rm d}t 
\leq \|\psi\|_k \int_0^{\infty}e^{(M-\alpha)t}{\rm d}t 
<\infty.
\end{equation*} 
Hence 
$$(L^\l(f)+i\alpha)^{-1}\psi \in \H^k = D((L^\l_0)^k).$$
A similar argument shows that $(L^\l(f)+i\alpha)^{-1}\psi \in 
D((L^\l_0)^k)$ also if $\alpha < -M$ completing the proof. 
\end{proof}
In the following, for every differentiable function $f$ on $S^1$, we shall 
denote by $f'$ the function on $S^1$ defined by 
$f'(e^{i\theta})=\frac{{\rm d}}{{\rm d}\theta}f(e^{i\theta})$. Moreover 
if $f$ is any integrable function on $S^1$ we shall use the notation  
$\int_{S^1}f$ for the integral 
$\int_{-\pi}^{\pi} f(e^{i\theta})d\theta .$ 
\begin{lemma} 
\label{supercommutatorsLemma}
Let $\psi$ be a vector in the domain of $(L^\l_0)^2$ and let 
$f$ be a real smooth function on $S^1$. Then the following hold: 
\begin{itemize}
\item[$(i)$] $L^\l(f)\psi \in D(Q_\l)$ and 
$$Q_\l L^\l(f)\psi= L^\l(f) Q_\l\psi + \frac{i}2 G^\l(f')\psi.$$ 

\item[(ii)] $G^\l(f)\psi \in D(Q_\l)$ and
$$Q_\l G^\l(f)\psi= -G^\l(f) Q_\l\psi + 
2 L^\l(f)\psi -\frac{c}{24\pi}\left(\int_{S^1}f \right)\psi.$$

\item[$(iii)$] $G^\l(f)\psi \in D(G^\l(f))$ and 
$$G^\l(f)^2\psi= 
L^\l(f^2)\psi + \frac{c}{12\pi}\left(\int_{S^1}(f'^2-\frac14 f^2)\right) 
\psi.$$  
\end{itemize}
\end{lemma}
\begin{proof} These are rather straightforward consequences of 
the Ramond algebra relations in Eq. (\ref{svirdef}) and of the energy 
bounds in Eq. (\ref{e-boundsL}) and Eq. (\ref{e-boundsG}) together with 
the fact that $V_\l$ is a core for every power of $L^\l_0$. \end{proof}
\begin{proposition} 
\label{deltaResolvent1}
Let $f$ be a real smooth function on $S^1$. If 
$\alpha \in \RR$ and  $|\alpha  |$ is sufficiently large then, 
for every $\psi \in D((L^\l_0)^2)$, 
$(L^\l(f)+i\alpha)^{-1}\psi \in D((L^\l_0)^2)$ and  
$$
Q_\l (L^\l(f)+i\alpha)^{-1}\psi = (L^\l(f)+i\alpha)^{-1}Q_\l\psi 
-\frac{i}2 
(L^\l(f)+i\alpha)^{-1}G^\l(f')(L^\l(f)+i\alpha)^{-1}\psi.
$$
\end{proposition}
\begin{proof} By Proposition \ref{resolvent} if $|\alpha |$ 
sufficiently large then $(L^\l(f)+i\alpha)^{-1}\psi \in D((L^\l_0)^2)$, 
for any $\psi \in D((L^\l_0)^2)$. Hence, by 
Lemma \ref{supercommutatorsLemma} $(i)$, 
$L^\l(f)(L^\l(f)+i\alpha)^{-1}\psi \in D(Q_\l)$ and
$$Q_\l L^\l(f)(L^\l(f)+i\alpha)^{-1}\psi= 
L^\l(f) Q_\l(L^\l(f)+i\alpha)^{-1}\psi + 
\frac{i}2 G^\l(f')(L^\l(f)+i\alpha)^{-1}\psi.$$ Adding 
$i\alpha  Q_\l (L^\l(f)+i\alpha)^{-1}\psi$ to both sides of the previous 
equality we find 
$$Q_\l\psi=
(L^\l(f) +i\alpha)Q_\l(L^\l(f)+i\alpha)^{-1}\psi +
\frac{i}2 G^\l(f')(L^\l(f)+i\alpha)^{-1}\psi,$$ 
so that 
$$(L^\l(f) +i\alpha)Q_\l(L^\l(f)+i\alpha)^{-1}\psi= Q_\l \psi 
- \frac{i}2 G^\l(f')(L^\l(f)+i\alpha)^{-1}\psi$$
and the conclusion follows by letting $(L^\l(f)+i\alpha)^{-1}$ 
act to both sides of the latter equality. 
\end{proof}
\begin{proposition}
\label{boundedProduct}
 Let $f_1$ and $f_2$ be real smooth functions on $S^1$ 
and assume that $f_1^2 \leq Cf_2$ for some $C>0$. Then, for any 
nonzero $\alpha \in \RR$ 
$$G^\l(f_1)(L^\l(f_2)+i\alpha)^{-1} \in B(\H_\l).$$
\end{proposition}
\begin{proof} 
Let $\beta\in \RR$. By Proposition \ref{resolvent}, if $|\beta|$ is 
sufficiently large  we have 
$$(L^\l(f)+i\beta)^{-1}D((L^\l_0)^2) \subset D((L^\l_0)^2)$$ 
and consequently 
$G^\l(f_1)(L^\l (f_2)+i\beta)^{-1}$ is densely defined. 
Moreover, 
$$\left( G^\l(f_1)(L^\l(f_2)+i\beta)^{-1}\right)^* \supset
 (L^\l(f_2)-i\beta)^{-1}G^\l(f_2)$$ 
is also densely defined and hence $G^\l(f_1)(L^\l(f_2)+i\beta)^{-1}$ is 
closable. 

From Lemma \ref{supercommutatorsLemma} $(iii)$ it follows that 
\begin{eqnarray*}
\|G^\l(f_1)(L^\l(f_2)+i\beta)^{-1}\psi\|^2 &=& 
((L^\l(f_2)+i\beta)^{-1}\psi,L^\l(f_1^2)(L^\l(f_2)+i\beta)^{-1}\psi) \\
&+& \frac{c}{12\pi}\left(\int_{S^1}(f_1'^2-\frac14 
f_1^2)\right) \|(L^\l(f_2)+i\beta)^{-1}\psi\|^2,
\end{eqnarray*}
for all $\psi \in D((L^\l_0)^2)$.
By assumption $Cf_2-f_1^2 \geq 0$ and hence, as a consequence of  
\cite[Theorem 4.1]{FH}, $L^\l(Cf_2-f_1^2)$ is bounded from below. 
It follows that there exists $\tilde{C}>0$ such that 

\begin{eqnarray*}
\|G^\l(f_1)(L^\l(f_2)+i\beta)^{-1}\psi\|^2 &\leq&
((L^\l(f_2)+i\beta)^{-1} 
\psi,(CL^\l(f_2)+\tilde{C})(L^\l(f_2)+i\beta)^{-1}\psi) 
\\
&\leq& 
\|(L^\l(f_2)-i\beta)^{-1}(CL^\l(f_2)+\tilde{C})(L^\l(f_2)+i\beta)^{-1} \|
\cdot \|\psi\|^2 \\
&\leq & \left( \frac{C}{2|\beta|} +\frac{\tilde{C}}{|\beta |^2} 
\right)
\|\psi\|^2,
\end{eqnarray*}
for all $\psi \in D((L^\l_0)^2)$. Therefore 
$G^\l(f_1)(L^\l(f_2)+i\beta)^{-1}$ restricts to a bounded linear map on 
$D((L^\l_0)^2)$ and, since it is closable, it must be bounded on its domain. 
Moreover, since $(L^\l(f_2)+i\beta)^{-1}$ belongs to $B(\H_\l)$ and 
$G^\l(f_1)$ is closed, $G^\l(f_1)(L^\l(f_2)+i\beta)^{-1}$ is closed. 
Accordingly $G^\l(f_1)(L^\l(f_2)+i\beta)^{-1} \in B(\H_\l)$. 
Now, if $\alpha \in \RR$ and $\alpha\neq 0$, the operator 
$(L^\l(f_2)+i\beta)(L^\l(f_2)+i\alpha)^{-1}$ belongs to $B(\H_\l)$. 
Hence 
\begin{eqnarray*}
G^\l(f_1)(L^\l(f_2)+i\alpha)^{-1} &=& G^\l(f_1)(L^\l(f_2)+i\beta)^{-1}
(L^\l(f_2)+i\beta)(L^\l(f_2)+i\alpha)^{-1} \\
&\in & B(\H_\l).  
\end{eqnarray*}
\end{proof}  
\begin{lemma} 
\label{L'HospitalLemma}
Let $f$ be a real smooth function on $S^1$ such that 
${\rm supp} f \subset \overline{I}$ and $f(z)>0$ for all $z\in I$, 
for some interval $I \in \I$. Assume moreover that $f'(z)\neq 0$ 
for all $z\in I$ sufficiently close to the boundary. 
Then there exists $C>0$ such that 
$f'^2 \leq C f$.  
\end{lemma} 
\begin{proof} Let $h$ be the real function on $S^1$ defined by 
\[
 h(z)\equiv \left\{ \begin{array}{cc} 0 & {\rm if}\; z \in \overline{I'}, 
\\
\frac{f'^2(z)}{f(z)} & {\rm if} \; z \in I.
\end{array}   
\right.
\] 
Clearly $h$ is continuous at every point of $I\cup I'$ and the 
restriction of $h$ to $I'$ is continuous. Now let $\zeta$ be 
a boundary point of $I$ and let $z_n$ be a sequence in $I$ converging 
to $\zeta$. Then, by L'Hospital's rule,   
\begin{eqnarray*}
\lim_{n \to \infty}h(z_n) &=& \lim_{n \to \infty}\frac{f'^2(z_n)}{f(z_n)}
= \lim_{n \to \infty}\frac{2f'(z_n)f''(z_n)}{f'(z_n)} \\
&=& 2f''(\zeta) =0=h(\zeta).  
\end{eqnarray*}
It follows that $h$ is continuous on $S^1$ and consequently it is bounded 
from above by some constant $C>0$.  
Then $f'^2=hf \leq Cf$. 
\end{proof}
\begin{theorem} 
\label{deltaResolvent2}
Let $\alpha$ be a real number and let $f$ be 
as in Lemma \ref{L'HospitalLemma}. Then if $|\alpha |$ is sufficiently 
large, $(L^\l(f)+i\alpha)^{-1} \in D(\delta)$ and 
\[
\delta\big((L^\l(f)+i\alpha)^{-1} \big)= 
-\frac{i}2 (L^\l(f)+i\alpha)^{-1}G^\l(f')(L^\l(f)+i\alpha)^{-1}.
\]
\end{theorem}
\begin{proof}
Let $\alpha$ be any nonzero real number and let 
$$b\equiv -\frac{i}2 (L^\l(f)+i\alpha)^{-1}G^\l(f')(L^\l(f)+i\alpha)^{-1}.$$ 
By Lemma \ref{L'HospitalLemma} and Proposition \ref{boundedProduct}, 
$G^\l(f')(L^\l(f)+i\alpha)^{-1} \in B(\H_\l)$ and hence 
$b \in B(\H_\l)$. If $|\alpha|$ is sufficiently large then, 
by Proposition \ref{deltaResolvent1} and the fact that 
$(L^\l(f)+i\alpha)^{-1}$ is even (it commutes with $\Gamma_\l$), we have
$$
Q_\l (L^\l(f)+i\alpha)^{-1}\psi = 
\gamma_s\big((L^\l(f)+i\alpha)^{-1}\big)Q_\l\psi + b\psi, 
$$ 
for all $\psi \in D((L^\l_0)^2)$ and since $ D((L^\l_0)^2)$ is a core
for $Q_\l$ 
the conclusion follows from Lemma \ref{corelemma}.
\end{proof}
\begin{proposition}
\label{deltaGResolvent1}
Let $f$ be a real smooth function on $S^1$. If
$\alpha \in \RR$ and  $|\alpha  |$ is sufficiently large then,
for every $\psi \in D((L^\l_0)^2)$, 
\begin{eqnarray*}
(L^\l(f)+i\alpha)^{-1}\psi \in D((L^\l_0)^2), \; 
(L^\l(f)+i\alpha)^{-1}G^\l(f')(L^\l(f)+i\alpha)^{-1}\psi \in 
D(L^\l_0), \\
G^\l(f)(L^\l(f)+i\alpha)^{-1}\psi \in D(Q_\l)
\end{eqnarray*}
and
\begin{eqnarray*}
Q_\l G^\l(f)(L^\l(f)+i\alpha)^{-1}\psi &=& 
-G^\l(g)(L^\l(f)+i\alpha)^{-1}Q_\l\psi \\
&+& \left( 2 L^\l(f) -\frac{c}{24\pi}\int_{S^1}f  \right) 
(L^\l(f)+i\alpha)^{-1}\psi \\
&+& \frac{i}2G^\l(f)
(L^\l(f)+i\alpha)^{-1}G^\l(f')(L^\l(f)+i\alpha)^{-1}\psi.
\end{eqnarray*}
\end{proposition}

\begin{proof}
It follows from Proposition \ref{resolvent} that if $|\alpha |$ is 
sufficiently large then 
$$
(L^\l(f)+i\alpha)^{-1} D((L^\l_0)^k) \in D((L^\l_0)^k),\; {\rm for} 
\; k=1,2.
$$
Now let $\psi \in D((L^\l_0)^2)$ so that 
$(L^\l(f)+i\alpha)^{-1}\psi \in D((L^\l_0)^2)$.  
By Lemma \ref{supercommutatorsLemma} $(ii)$, we have
\begin{eqnarray*}
Q_\l G^\l(f)(L^\l(f)+i\alpha)^{-1}\psi &=& 
-G^\l(f) Q_\l(L^\l(f)+i\alpha)^{-1} \psi \\
&+& \left(2 L^\l(f) -\frac{c}{24\pi}\int_{S^1}f 
\right)(L^\l(f)+i\alpha)^{-1}\psi.  
\end{eqnarray*}

Moreover, by Proposition \ref{deltaResolvent1} we have 
$$
Q_\l (L^\l(f)+i\alpha)^{-1}\psi = (L^\l(f)+i\alpha)^{-1}Q_\l\psi
-\frac{i}2
(L^\l(f)+i\alpha)^{-1}G^\l(f')(L^\l(f)+i\alpha)^{-1}\psi.
$$

From the fact that $(L^\l(f)+i\alpha)^{-1}\psi \in D((L^\l_0)^2)$ and 
$Q_\l \psi \in  D(L^\l_0)$ we have that $Q_\l (L^\l(f)+i\alpha)^{-1}\psi$
and $(L^\l(f)+i\alpha)^{-1}Q_\l\psi$ belong to $D(L^\l_0)$. 
Hence 
$$(L^\l(f)+i\alpha)^{-1}G^\l(f')(L^\l(f)+i\alpha)^{-1}\psi$$ 
also belongs to $D(L^\l_0) \subset D(G^\l(f))$ and 
\begin{eqnarray*}
G^\l(f)Q_\l (L^\l(f)+i\alpha)^{-1}\psi &=& 
G^\l(f)(L^\l(f)+i\alpha)^{-1}Q_\l\psi \\
&-& \frac{i}2
G^\l(f)(L^\l(f)+i\alpha)^{-1}G^\l(f')(L^\l(f)+i\alpha)^{-1}\psi.
\end{eqnarray*}
It follows that
\begin{eqnarray*}
Q_\l G^\l(f)(L^\l(f)+i\alpha)^{-1}\psi &=&
-G^\l(f)(L^\l(f)+i\alpha)^{-1}Q_\l\psi \\
&+& \frac{i}2
G^\l(f)(L^\l(f)+i\alpha)^{-1}G^\l(f')(L^\l(f)+i\alpha)^{-1}\psi
\\
&+& \left(2 L^\l(f) -\frac{c}{24\pi}\int_{S^1}f
\right)(L^\l(f)+i\alpha)^{-1}\psi.
\end{eqnarray*}
\end{proof}
\begin{lemma} 
\label{boundedProduct2}
Let $\alpha$ be a nonzero real number and let $f$ be a real nonnegative 
smooth function on $S^1$. Then 
$G^\l(f)(L^\l(f)+i\alpha)^{-1} \in B(\H_\l)$. 
\end{lemma}
\begin{proof}
Since $f$ is continuous on $S^1$ then it is bounded from above by some 
constant $C>0$. Accordingly $f^2\leq Cf$ and the conclusion follows from 
Proposition \ref{boundedProduct}.
\end{proof}
\begin{theorem}
\label{deltaResolvent4}
Let $\alpha$ be a real number and let $f$ be
as in Lemma \ref{L'HospitalLemma}. Then, if $|\alpha |$ is sufficiently
large, $G^\l(f)(L^\l(f)+i\alpha)^{-1} \in D(\delta)$ and
\begin{eqnarray*}
\delta\big(G^\l(f)(L^\l(f)+i\alpha)^{-1} \big) &=&
\left( 2 L^\l(f) -\frac{c}{24\pi}\int_{S^1}f  \right)
(L^\l(f)+i\alpha)^{-1} \\
&+& \frac{i}2G^\l(f)(L^\l(f)+i\alpha)^{-1}G^\l(f')(L^\l(f)+i\alpha)^{-1}.
\end{eqnarray*}
\end{theorem}
\begin{proof} Let $\alpha$ be any nonzero real number. We denote 
$G^\l(f)(L^\l(f)+i\alpha)^{-1}$ by $a$ and 
$$
\left( 2 L^\l(f) -\frac{c}{24\pi}\int_{S^1}f  \right)
(L^\l(f)+i\alpha)^{-1} 
+ \frac{i}2G^\l(f)(L^\l(f)+i\alpha)^{-1}G^\l(f')(L^\l(f)+i\alpha)^{-1}
$$
by $b$. From Lemma \ref{boundedProduct2} we know that $a\in B(\H_\l)$
and it is easy to see that $\gamma_\l(a)=-a$. It is also evident that
$$
\left( 2 L^\l(f) -\frac{c}{24\pi}\int_{S^1}f  \right)
(L^\l(f)+i\alpha)^{-1} \in B(\H_\l). 
$$
By Lemma \ref{L'HospitalLemma} and Proposition \ref{boundedProduct} 
we also know that 
$G^\l(f')(L^\l(f)+i\alpha)^{-1}$ belongs to $B(\H_\l)$. As a consequence 
$b \in B(\H_\l)$. Now, if $|\alpha|$ is sufficiently large, it follows 
from Proposition \ref{deltaGResolvent1}, that $a\psi\in D(Q_\l)$ and 
$$Q_\l a \psi =\gamma_\l(a)Q_\l \psi + b\psi,$$
for all $\psi \in D((L^\l_0)^2)$. The conclusion then 
follows from Lemma \ref{corelemma} because $D((L^\l_0)^2)$ is a core for
$Q_\l$.     
\end{proof}

\begin{lemma} 
\label{dense(I)lemma}
$\A_\l(I) \cap D(\delta)$ is a strongly dense unital 
*-subalgebra of $\A_\l(I)$ for all $I\in \I$. 
\end{lemma}
\begin{proof}
By Proposition \ref{proposition*subalgebras} $\A_\l(I) \cap D(\delta)$
is a unital *-subalgebra of $\A_\l(I)$ and hence, by von Neumann density 
theorem, it is enough to show that 
$$\left(\A_\l(I) \cap D(\delta)\right)' \subset \A_\l(I)'.$$ 
To this end let 
$f$ be an arbitrary real smooth function on $S^1$ with support in $I$. 
Recalling that $I$ must be open it is easy to see that there is an 
interval $I_0 \in \I$ such that  $\overline{I_0} \subset I$ and 
${\rm supp} f \subset I_0$  and a smooth function
$g$ on $S^1$ such that ${\rm supp} g\subset \overline{I_0}$, 
$g(z)>0$ for all $z \in I_0$, $g'(z)\neq 0$ for all $z\in I_0$ 
sufficiently close to the boundary and $g(z)=1$ for all $z \in {\rm supp} 
f$.  Accordingly, there is a real number $s>0$ such that 
$f(z)+sg(z) > 0$ for all $z \in I_0$. Now let $f_1=f+sg$ and 
$f_2=sg$. Then $f=f_1-f_2$. Moreover $f_1$ and $f_2$ satisfy the 
assumptions in 
Lemma \ref{L'HospitalLemma} and have support in $I$. 
Hence it follows from Theorem \ref{deltaResolvent2}, Theorem 
\ref{deltaResolvent4} and the definition of $\A_\l(I)$ that 
there exists a nonzero real number $\alpha$ such that the 
operators $(L^\l(f_i)+i\alpha)^{-1}$ and 
$G^\l(f_i)(L^\l(f_i)+i\alpha)^{-1}$, $i=1,2$, belong to 
$\left(\A_\l(I) \cap D(\delta)\right)$. 
As a consequence if $a \in \left(\A_\l(I) \cap D(\delta)\right)' $.
Then $a$ commutes with $L^\l(f_i)$ and $G^\l(f_i)$, $i=1,2$. Therefore, 
if $\psi_1, \psi_2 \in C^\infty(L^\l_0)$ then, 
\begin{eqnarray*}
(a\psi_1,L^\l(f)\psi_2) &=& 
(a\psi_1,L^\l(f_1)\psi_2)-(a\psi_1,L^\l(f_2)\psi_2) \\
&=& (aL^\l(f_1)\psi_1,\psi_2)-(aL^\l(f_2)\psi_1,\psi_2) \\
&=& (aL^\l(f)\psi_1,\psi_2)
\end{eqnarray*} 
and, since $C^\infty(L^\l_0)$ is a core for $L^\l(f)$, it follows that
$a$ commutes with $L^\l(f)$ and hence with $e^{iL^\l(f)}$. Similarly 
$a$ commutes with $e^{iG^\l(f)}$. Hence $a \in \A_\l(I)'$ and the 
conclusion follows.
\end{proof} 
Now we can state and prove the main result of this section. 
\begin{theorem}\label{th:dense-R}
$\gA_\l(I)$ is a strongly dense unital *-subalgebra of $\A_\l(I)$ 
for all $I\in \I$. 
\end{theorem}
\begin{proof}
Let $I_0 \in \I$ be an interval whose closure $\overline{I_0}$ is 
contained in $I$ and let $a \in \A_\l(I_0)\cap D(\delta)$. Now, 
if the support of the function $f \in C^\infty_c(\RR)$ is sufficiently 
close to $0$ then 
$$a_f=\int_\RR e^{itL^\l_0}ae^{-itL^\l_0}f(t){\rm d}t \in \A_\l(I).$$ 
Moreover, by Proposition 
\ref{propositionA_f}, $a_f\in C^\infty(\delta)$ and thus 
$a_f \in \gA_\l(I)$. It follows that 
$\A_\l(I_0)\cap D(\delta) \subset \gA_\l(I)''$ and by 
Lemma \ref{dense(I)lemma} that $\A_\l(I_0) \subset \gA_\l(I)''$
By conformal covariance we have 
\[
\A_\l(I)=\bigvee_{\overline{I_0}\subset I }\A_\l(I_0)
\]
and hence $\A_\l(I) \subset {\gA}_\l(I)''$.
\end{proof}
We have thus proved the following.
\begin{theorem}\label{main1}
Let $\lambda$ be a unitary, graded, positive energy representation of the 
Ramond algebra and denote as above by $L_n^\lambda, G_r^\lambda$, 
$n,r\in\mathbb Z$, the Virasoro elements and the Fermi elements. 
Assume that ${\rm Tr}(e^{- \beta L^\l_0}) <\infty$ for all $\beta >0$.
Then, with $\A_\lambda$ the associated net of local von Neumann algebras 
on $S^1$, 
we have a net of graded, $\theta$-summable spectral 
triples (in fact a net of quantum algebras) 
$(\gA_\lambda, \H_\lambda, Q_\lambda)$ where $Q_\lambda \equiv G^\lambda_0$ 
such that $\gA_\l(I)$ is a strongly dense unital $^*$-subalgebra of 
$\A_\lambda(I)$ for every interval $I\subset S^1$ in $\I$.

In particular this is the case if $\l$ is the irreducible unitary representation 
with central charge $c$ and lowest weight $h_\lambda = c/24$ (minimal lowest weight) 
and the Fredholm index is equal to 1.
\end{theorem}
\begin{remark} If the graded unitary positive 
energy representation $\l$ of the Ramond algebra is a direct of 
sum of finitely many irreducible (not necessarily graded) 
subrepresentations then ${\rm Tr}(e^{- \beta L^\l_0}) <\infty$ for all 
$\beta >0$ and hence the above theorem applies. The same is true 
also for certain infinite direct sums of irreducibles.  
\end{remark}
\begin{remark}
The irreducible unitary representations of the Ramond algebra with lowest 
weight $h_\lambda \neq c/24$ are not graded. Nonetheless in this case the 
above Theorem \ref{main1} gives so-called {\em odd} spectral triples.
\end{remark}
\section{Spectral triples from the Neveu-Schwarz algebra}\label{Sect:NS}

The construction in the preceding section can be adjusted to obtain spectral 
triples from representations of the Neveu-Schwarz algebra. 
The essential difficulty in this 
case is that, while the Ramond algebra contains the global supercharge operator
 $Q=G_0$, this is not true for the Neveu-Schwarz algebra: here one may define 
$\delta$ by abstract commutation relations, but then one is soon faced with 
the question whether this formal superderivation still has a nontrivial 
domain as in Sect. \ref{SectSuperDerivations}. We shall overcome this problem 
here below.

At this point we should make a comment. The nets of generalised quantum 
algebras that we shall construct provide an intrinsic structure visible 
in any representation, both in the Ramond and in the Neveau-Schwarz 
case. Indeed if a given representation of the Neveu-Schwarz algebra is 
``locally normal'' with respect to a reprensentation of the Ramond 
algebra with the same central charge, as it is natural to expect
(but difficult to prove in general), then one could carry the generalised 
quantum algebra from one representation to the other one. It is however 
unclear to us that one can naturally associate a cyclic cocycle to any net of 
generalised quantum algebras. In this respect the interesting 
representations so far appear to be the representations of the Ramond 
algebra.

The {\em Neveu-Schwarz algebra} is the super-Lie algebra generated by even 
elements $L_n$, $n\in\mathbb Z$, odd elements $G_r$, $r \in \ZZ +1/2$, and a 
central even element $k$, satisfying the relations (\ref{svirdef}). 
In other words the commutation relations of the Ramond algebra hold also 
here but the index $r$ runs through $\mathbb Z + 1/2$.

We shall consider representations $\lambda$ of the Neveu-Schwarz algebra by 
linear 
endomorphisms, denoted by $L_m^\l, G_r^\l, k^\l$, $m \in \ZZ$, 
$r \in \ZZ + 1/2$, of a complex vector 
space $V_\lambda$ equipped with an involutive linear endomorphism $\Gamma_\l$ 
inducing the super-Lie algebra grading. The endomorphisms 
$L^\l_m, G^\l_r, k^\l$, satisfy \eqref{svirdef} with respect to the brackets 
given by the super-commutator induced by $\Gamma_\l$, and we suppose 
they satisfy the properties corresponding to $(i)-(iv)$ as stated for the 
Ramond algebra in the preceding section.
Moreover we assume that $\l$ is a positive energy representation, 
namely that $(v,L^\l_0 v)\geq 0$ for all $v \in V_\l$. 
(For the representations of the Ramond algebra 
positivity of the energy was a consequence of $(i)-(iv)$ but this is not the 
case for the Neveu-Schwarz algebra.)
If $\l$ is an irreducible unitary positive energy representation with 
lowest weight $h_\l$ then, in contrast with the Ramond case,  
it is automatically graded by $\Gamma_\l =e^{i 2\pi (L^\l_0-h_\l}$ 
and in fact it satisfies all the above assumptions.  

As in Section \ref{Sect:R}, the elements 
$L^\l_m, G^\l_r$ define closable operators on the Hilbert space completion 
$\H_\l$ of $V_\l$ and their closure is denoted by 
the same symbol. We also have the linear energy bounds  \eqref{e-boundsL} and 
\eqref{e-boundsG}. Note however, that, in contrast with the Ramond case, the unitary positive energy 
representations (not necessarily irreducible) are automatically graded by 
$\Gamma_\l =e^{i 2\pi L^\l_0}$.

For the Fermi Neveu-Schwarz fields we shall consider a Fourier expansion with 
respect to a different basis as follows. 
Let $f$ be a smooth function on $S^1$ with support contained in some interval
$I \in \I_0$, where $\I_0$ is defined as in Section \ref{SectSpectralTriples}. 
This is equivalent to require that supp$f$ does not contain the point $-1$. 
The Fourier coefficients here are
\begin{equation*}
\hat{f}_r=\int_{-\pi}^\pi f(e^{i\theta})e^{-ir\theta}\frac{{\rm d}\theta}{2\pi},\quad
r\in \ZZ +\frac{1}{2},
\end{equation*}
and they are rapidly decreasing (to this end it is crucial that the support of 
$f$ does not contains $-1$). Then because of the linear energy bounds the map 
\begin{equation*}
V_\l \ni v \mapsto \sum_{r \in \ZZ +1/2}\hat{f}_rG^\l_r v
\end{equation*}
defines a closable operator $G^\l(f)$ on $\H_\l$ whose closure is denoted by 
$G^\l(f)$ again; $L^\l(f)$ is defined as in the preceding section.
The domains of $G^\l(f)$ and $L^\l(f)$ contain $D(L^\l_0)$ and they 
leave invariant $C^\infty(L^\l_0)$. Moreover, if $f$ is real, $L^\l(f)$ and $G^\l(f)$ are 
selfadjoint operators and their restriction to  any core for 
$L^\l_0$ are essentially selfadjoint operators cf. \cite{BS-M}. Actually 
in the case of $L^\l(f)$ the above properties hold without any restriction 
on the support of $f\in C^\infty (S^1)$. 

We can then define a net ${\A}_\l$ of von Neumann algebras on 
$S^1\setminus\{-1\}$ by
\begin{equation} 
\A_\l(I) \equiv \{ e^{iL^\l(f)}, e^{iG^\l(f)}:f\in C^{\infty}(S^1) 
\;{\rm real},\, {\rm supp}f \subset I\}'', \;I \in \I_0.
\end{equation}

As in \cite[Sect. 6.3]{CKL} it can be shown that $\A_\l$ extends to a 
graded-local conformal covariant net on $S^1$ (in general 
without vacuum vector). In fact, we have 
\begin{equation} 
e^{i4\pi L^\l_0} \in \bigcap_{I\in \I}\A_\l(I)'
\end{equation}
so that the action of $\Diff^{(\infty)}(S^1)$ on the local algebras 
factors through $\Diff^{(2)}(S^1)$.  

Using the definition of smeared fields and the (anti-) commutation relations of 
the Neveu-Schwarz algebra, we get:
\begin{proposition}\label{prop:CRs}
 Let $f,g$ be smooth functions on $S^1$ with support in some  
$I \in \I_0$. Then the smeared fields of the Neveu-Schwarz 
algebra satisfy the following (anti-) commutation relations 
on the common invariant core $C^\infty (L^\l_0)$:
\begin{equation}
\begin{gathered}\label{svirCRs}
    [L^\l(f), L^\l(g)] = -i L^\l(f'g)+iL^\l(fg') + 
i\frac{c}{24\pi} \int_{S^1} (f'''g+f'g),\\
    [L^\l(f), G^\l(g)] = i G^\l(f g') -\frac{i}{2}G^\l( f'g),\\
    [G^\l(f), G^\l(g)] = 2L^\l (fg) +
\frac{c}{6\pi} \int_{S^1} (f'g' - \frac14 fg).
\end{gathered}
\end{equation}
\end{proposition}
\begin{remark} The above relation also holds, without any restriction on 
the supports of the smooth functions $f$ and $g$, when $\l$ is a 
representation of the Ramond algebra as in Section \ref{Sect:R}.  
\end{remark}
Now let $\varphi$ be any real smooth function on $S^1$ with support in some 
interval in $\I_0$. Then $G^\l(\varphi)$ is an odd self-adjoint operator on the 
graded Hilbert space $\H_\l$ and hence we can define as in Section 
\ref{SectSuperDerivations} a corresponding superderivation 
$\delta_\varphi = [G^\l(\varphi),\, \cdot\, ]$ on $B(\H_\l)$ which, by $(i)$ 
and 
$(ii)$ in Proposition \ref{delta} is odd and antisymmetric. 

We now make the following observation: with $\varphi$ any function from the 
subset $\C_I \subset C^\infty(S^1, \RR)$ defined by 
\begin{equation}\label{eq:phi_I}
\C_I \equiv \{ \varphi \in C^\infty (S^1, \RR):\varphi(z) = 1 \;\forall 
z\in I,-1 \notin {\rm supp} \varphi \}
\end{equation}
and $f$ a smooth function with support in $I$, we obtain form 
Proposition \ref{prop:CRs} the relations $(i)-(ii)$ of Lemma 
\ref{supercommutatorsLemma} with $Q_\l$ replaced by $G^\l(\varphi)$
(the precise domain statements follow again from the linear energy bounds). 
Therefore, for any $\varphi \in \C_I$, we may interpret $G(\varphi)$ as a 
{\em local supercharge} for  $\A_\l(I)$. Actually, as a consequence of the 
following proposition we, we will be able to use these local supercharges to define 
a net of superderivations which has the desired commutation relations. 
\begin{proposition}\label{prop:delta-welldefined} Let 
$\varphi, \tilde{\varphi} \in \C_I$, $I\in \I_0$, and let $\delta_\varphi$ and 
$\delta_{\tilde{\varphi}}$ the superderivations on $B(\H_\l)$ associated 
to the selfadjoint operators $G^\l(\varphi)$ and $G^\l(\tilde{\varphi})$
respectively. Then the following hold:   

\begin{itemize}
 \item[$(i)$] $D(\delta_{\varphi})\cap\A_\l(I)= 
D(\delta_{\tilde{\varphi}})\cap \A_\l(I)$ and 
$\delta_\varphi(a)=\delta_{\tilde{\varphi}}(a)$ for all 
 $a \in D(\delta_{\varphi})\cap\A_\l(I)$. 
\item[$(ii)$] $\delta_\varphi(a) \in \A_\l (I)$ for all  $a \in 
D(\delta_{\varphi})\cap\A_\l(I)$.
\end{itemize}
\end{proposition}
\begin{proof}
$(i)$ It is enough to show that if $a\in D(\delta_\varphi)\cap \A_\l(I)$ 
then $a\in D(\delta_{\tilde{\varphi}})\cap \A_\l(I)$ and 
$\delta_{\tilde{\varphi}}(a) = \delta_{\varphi}(a)$.
Given any $a\in D(\delta_\varphi)\cap \A_\l(I)$ and any 
$\psi_1  \in C^\infty(L^\l_0)$, we have $a\psi_1 \in D(G(\varphi))$  and 
\begin{equation*}
 \begin{aligned}
  (a\psi_1,G^\l(\tilde{\varphi})\psi_2) &= 
(a\psi_1,G^\l(\varphi)\psi_2)+ (a\psi_1,G^\l(\tilde{\varphi}-\varphi)\psi_2) \\
&= (G^\l(\varphi)a\psi_1,\psi_2) + 
(a\psi_1, G^\l(\tilde{\varphi}-\varphi)\psi_2), 
\end{aligned}
\end{equation*}
for all $\psi_2 \in C^\infty(L^\l_0)$.
Now, from the fact that $\tilde{\varphi}-\varphi$ vanishes on $I$ it follows 
that $G^\l(\tilde{\varphi}-\varphi)$ is affiliated with $\A_\l(I')$ and hence,
by graded locality for the net $\A_\l$, it is also affiliated with 
$Z_\l\A_\l(I)'Z_\l^*$, where as before $Z_\l=(1-i\Gamma_\l)/(1-i)$. Accordingly
$a\psi_1 \in D(G^\l(\tilde{\varphi}-\varphi))$ and 
$$ G^\l(\tilde{\varphi}-\varphi)a\psi_1 = 
\gamma_\l(a)G^\l(\tilde{\varphi})\psi_1 -\gamma_\l(a)G^\l(\varphi)\psi_1 $$
so that, recalling that $G^\l(\varphi -\tilde{\varphi})$ is selfadjoint, we have 
$$
(a\psi_1,G^\l(\tilde{\varphi})\psi_2) = 
(\delta_\varphi(a)\psi_1 + \gamma_\l(a)G^\l(\tilde{\varphi})\psi_1,\psi_2).
$$
As $C^\infty(L^\l_0)$ is a core for $G^\l(\tilde{\varphi})$, it follows that 
$a\psi_1 \in D(G^\l(\tilde{\varphi}))$ and  
$$G^\l(\tilde{\varphi})a\psi_1= 
\delta_\varphi(a)\psi_1 + \gamma_\l(a)G^\l(\tilde{\varphi})\psi_1.$$
Since $\psi_1$ was an arbitrary vector in $C^\infty(L^\l_0)$  and the latter is a 
core for $G^\l(\tilde{\varphi})$, the conclusion follows using Lemma \ref{corelemma}. 

$(ii)$ Let $I_1$ be any interval in $\I_0$ 
containing the closure $\overline{I}$ of $I$.    
By $(i)$ we can assume  that ${\rm supp}\varphi \subset I_1$. 
Then for any $b \in \A_\l  (I_1)'$ and $\psi \in C^\infty (L^\l_0)$, we have
\begin{equation*}
  \begin{aligned}
  b \delta_\varphi(a) \psi = & b G^\l(\varphi) a \psi - b \gamma(a) 
G^\l(\varphi)\psi \\
= &  G^\l(\varphi) a b \psi -  \gamma(a) G^\l(\varphi) b \psi  = 
\delta_\varphi(a) b\psi 
  \end{aligned}
\end{equation*}
because $a \in \A_\l (I)$ and $G^\l(\varphi)$ is affiliated with $\A  
(I_1)$. 
So $\delta_\varphi(a) \in \A_\l (I_1)$. 
Since $I_1 \supset \overline{I}$ was arbitrary we obtain 
\[\delta_\varphi(a) \in \bigcap_{I_1 \supset \overline{I}} \A_\l(I_1) = 
\A_\l(I),\]
where the last equality is a consequence of conformal covariance of the net 
$\A_\l$.
\end{proof}
Now for for all $I \in \I_0$ and $\varphi \in \C_I$ we consider the 
unital *-algebra $\gA_\l(I) \equiv 
C^\infty(\delta_\varphi)\cap\A_\l(I)$ and the antisymmetric 
odd superderivation $\delta_I:\gA_\l(I) \mapsto \gA_\l(I)$ defined by 
$\delta_I \equiv \delta_\varphi|_{\gA_\l(I)}$, which do not depend 
on the choice of $\varphi \in \C_I$ and thus are well-defined. Accordingly, if $I_1, I_2 \in \I_0$, 
$I_1 \subset I_2$ and $\varphi \in \C_{I_2}\subset \C_{I_1}$ then  
\begin{eqnarray*}
\gA_\l(I_1) & = & C^\infty(\delta_\varphi)\cap\A_\l(I_1) \\
& \subset &  C^\infty(\delta_\varphi)\cap\A_\l(I_2) \\
& = & \gA_\l(I_2).                  
\end{eqnarray*}
Moreover, $\delta_{I_2}|_{I_1}=\delta_{I_1}$. 
Therefore the map $\I_0 \ni I \mapsto \gA_\l$ defines a net of unital 
*-algebras on $S^1\setminus\{-1\}$ and the map  $\delta ^\l: \I_0 \ni I \mapsto \delta_I$ 
is a net of $\sigma$-weakly closable antisymmetric odd
superderivations of $\gA_\l$. Moreover, it can be shown that it extends to a net on the double cover $S^{1(2)}$.
\begin{proposition}\label{Prop:Hamiltonian} Let $I \in \I_0$ and 
$a\in \gA_\l(I)$. Then, for every $\psi \in D(L^\l_0)$, $a\psi \in D(L^\l_0)$
and $L^\l_0a\psi-aL^\l_0\psi=\delta_I^2(a)\psi$. 
\end{proposition}
\begin{proof}
A closure argument shows that it is enough to prove the proposition for all
$\psi \in C^\infty(L^\l_0)$. 
Let $\varphi \in \C_I$ and let $\delta_\varphi=[G^\l(\varphi),\, \cdot \,]$ be 
the
corresponding superderivation on $B(\H_\l)$. 
By assumption $a \in D(\delta_\varphi^2)$ and hence, by (the proof of) 
Lemma \ref{lemma:deltaSquare}, 
for any $\psi \in C^\infty(L^\l_0)$ we have $a\psi \in D(G^\l(\varphi)^2)$ 
and $$G^\l(\varphi)^2a\psi-aG^\l(\varphi)^2\psi=\delta_I^2(a)\psi.$$ 
Now, by Proposition \ref{prop:CRs} we have 
$$G^\l(\varphi)^2\psi_1 = L^\l(\varphi^2)\psi_1+ 
\frac{c}{12\pi}\int_{S^1}(\varphi'^2 -\frac{1}{4}\varphi^2)\psi_1,$$
so
$$(a\psi, L^\l(\varphi^2)\psi_1) = (aL^\l(\varphi^2)\psi,\psi_1)
+(\delta_I^2(a)\psi,\psi_1),$$
for all $\psi_1 \in C^\infty(L^\l_0)$. Thus, $a\psi$ is 
in the domain of $L^\l(\varphi^2)$ and 
$$L^\l(\varphi^2)a\psi=aL(\varphi^2)\psi+ \delta_I^2(a)\psi.$$
Now, for all $\psi_1 \in C^\infty(L^\l_0)$ we have 
$$L^\l(\varphi^2)\psi_1+L^\l(1-\varphi^2)\psi_1=L^\l_0\psi_1$$
and since,  as a consequence of the fact that $1-\varphi^2$ 
vanishes on $I$, $L^\l(1-\varphi^2)$ is an (even) operator 
affiliated with $\A_\l(I')$, we also know that $a\psi$ is 
in the domain of $L^\l(1-\varphi^2)$
and $L^\l(1-\varphi^2)a\psi = aL^\l(1-\varphi^2)\psi.$ 
Accordingly
$$(a\psi, L^\l_0\psi_1) = (aL^\l_0\psi,\psi_1)
+(\delta_I^2(a)\psi,\psi_1),$$
so that $a\psi \in D(L^\l_0)$ and $L^\l_0a\psi=aL^\l_0\psi+\delta_I^2(a)\psi.$ 
\end{proof}
As a consequence of Proposition \ref{Prop:Hamiltonian} and of the 
discussion preceding it we can conclude that, 
provided that $e^{-\beta L^\l_0}$ is trace class for all 
$\beta>0$, $(\gA_\l,\H_\l, \delta_\l)$ is a net of generalized quantum 
algebras on $S^1\setminus\{-1\}$ with Hamiltonian $L^\l_0$ as defined in 
Section \ref{SectSpectralTriples}. 
Yet we do not know whether the algebras $\gA_\l(I)$ are dense in  
$\A_\l(I)$ or nontrivial at all, so let us now consider this point.
We fix any interval $I_0 \in \I_0$, any function $\varphi \in \C_{I_0}$
and consider the superderivation 
$\delta_{\varphi}=[G^\l(\varphi),\, \cdot \,]$ on $B(\H_\l)$ as above. 
We can apply the theory from Section \ref{Sect:R} to this setting again. 
Then one checks that all the statements from Proposition \ref{resolvent} 
through Lemma \ref{dense(I)lemma} hold true if we replace $Q_\l$ by 
$G^\l(\varphi)$, $\delta$ by $\delta_\varphi$ and 
consider only functions in $C^\infty (S^1)$ with support contained in 
$I_0$. In particular we have the following  
analogue of Lemma \ref{dense(I)lemma}. 
\begin{lemma}
Let $I_0$ be any interval in $\I_0$ and let $\varphi \in \C_{I_0}$. 
Then $D(\delta_\varphi)\cap \A_\l(I)$ is a strongly dense unital 
*-subalgebra of $\A_\l(I)$ for every $I \in \I_0$ such that 
$\overline{I} \subset I_0$. 
\end{lemma}
As $I_0$ was arbitrary, this will lead to an analogue of Theorem 
\ref{main1}, but first we need to adapt 
the essential ingredient from Section \ref{SectSuperDerivations} to the 
present situation, namely Proposition \ref{propositionA_f}.
As in Section \ref{Sect:R} for any $f\in C^\infty_c(\RR)$ and any
$a \in B(\H_\l)$ we set 
\begin{equation}
a_f \equiv \int_\RR e^{itL^\l_0}ae^{-itL^\l_0}f(t){\rm d}t. 
\end{equation}
\begin{proposition}\label{prop:local-smoothness} Let $I$ be any interval 
in $\I_0$ and let $\varphi \in \C_{I}$. Moreover let $I_0$ be any interval 
in $\I_0$ whose closure $\overline{I_0}$ is contained in $I$.  
Then there exists $\varepsilon >0$ such that for all 
$f\in C^\infty_c(\RR)$ with ${\rm supp}f \subset(-\varepsilon, \varepsilon)$ 
and all $a \in D(\delta_\varphi) \cap  \A (I_0)$ we have 
$a_f \in \gA_\l(I)$.   
\end{proposition}
\begin{proof}
We basically work as in the preceding sections. However, here we have to take 
care to remain ``local'' in order to preserve the right commutation 
relations. 

Since $\varphi \in \C_I$ its support is contained in some $I_1 \in \I_0$. 
Fix $\varepsilon$ such that $e^{it}I_1 \in \I_0$ and 
$e^{it}I_0 \subset I$ for all $t\in (-\varepsilon, \varepsilon)$.
Given $a \in D(\delta_\varphi) \cap \A_\l(I_0)$, by rotation covariance 
of the net $\A_\l$ we have   
\[
e^{itL^\l_0}ae^{-itL^\l_0}\in \A_\l(e^{it} I_0) \subset \A_\l(I), 
\]
for all $t\in (-\varepsilon, \varepsilon)$. Hence, if 
the support of the function $f \in C^\infty_c(\RR)$ is contained 
in $(-\varepsilon, \varepsilon)$  we also have $a_f \in \A_\l(I)$. 
Now from the definition of the smeared fields in the representation $\l$ 
it easily follows that 
$e^{itL^\l_0}G^\l(\varphi)e^{-itL^\l_0}=G^\l(\varphi_t)$ for all 
$t\in (-\varepsilon, \varepsilon)$, where the function $\varphi_t$ 
is defined by $\varphi_t(z)=\varphi(e^{-it}z)$.
Accordingly, for any $t\in (-\varepsilon, \varepsilon)$, 
$e^{itL^\l_0}ae^{-itL^\l_0} \in D(\delta_{\varphi_t}) \cap 
\A_\l(e^{it}I_0)$ and 
$$\delta_{\varphi_{t}}\big(e^{itL^\l_0}ae^{-itL^\l_0}\big)
=e^{itL^\l_0}\delta_\varphi(a)e^{-itL^\l_0}.$$
Now for all $t\in (-\varepsilon, \varepsilon)$, we have 
$\varphi_t, \varphi \in \C_{e^{it}I_0}$ and hence 
by Proposition \ref{prop:delta-welldefined} we can conclude
that $e^{itL^\l_0}ae^{-itL^\l_0}$ belongs to 
$D(\delta_\varphi)$ and that
\[
\delta_{\varphi}\big(e^{itL^\l_0}ae^{-itL^\l_0}\big)
= \delta_{\varphi_t}\big(e^{itL^\l_0}ae^{-itL^\l_0}\big)
\]
so that$$\delta_{\varphi}\big(e^{itL^\l_0}ae^{-itL^\l_0}\big)
=e^{itL^\l_0}\delta_\varphi(a)ae^{-itL^\l_0}.$$
It follows that $a_f \in D(\delta_\varphi)$ and 
$\delta_\varphi(a_f)=\delta_\varphi(a)_f$.

Next, for any $\psi \in C^{\infty}(L^\l_0 )$, we have 
$a_f\psi, \delta_\varphi(a)_f\psi\in C^{\infty}(L^\l_0 )$.
Hence $\delta_\varphi(a_f)\psi=\delta_\varphi(a)_f\psi \in 
D(G^\l(\varphi))$ and (cf. the proof of Proposition 
\ref{Prop:Hamiltonian} and Lemma \ref{lemmaA_f2})
\begin{eqnarray*}
G^\l(\varphi)\delta_\varphi(a_f)\psi & = & G^\l(\varphi)^2 a_f\psi- 
G^\l(\varphi)\gamma(a_f)G^\l(\varphi)\psi  \\
& = & G^\l(\varphi)^2 a_f\psi 
-\delta_\varphi(\gamma(a_f))G^\l(\varphi)\psi
-a_fG^\l(\varphi)^2\psi \\
& = & L^\l(\varphi_I^2)a_f\psi -a_fL(\varphi^2)\psi
-\delta_\varphi(\gamma(a_f))G^\l(\varphi) \psi \\ 
& = & L^\l_0 a_f\psi -a_fL^\l_0\psi 
-\delta_\varphi(\gamma(a_f))G^\l(\varphi) \psi \\
& = & ia_{f'}\psi -\delta_\varphi(\gamma(a_f))G^\l(\varphi)\psi \\
& = & ia_{f'}\psi + \gamma(\delta_\varphi(a_f))G^\l(\varphi)\psi. 
\end{eqnarray*} 
Thus $a_f \in D(\delta_\varphi^2)$ and 
$\delta_\varphi^2(a_f) =ia_{f'}$ by Lemma \ref{corelemma} and the 
conclusion easily follows by induction. 
\end{proof}
With the preceding modification of Proposition \ref{propositionA_f} the 
following final result is proved in the same manner as 
Theorem \ref{th:dense-R}.
\begin{theorem}
$\gA_\l(I)$ is a strongly dense unital *-subalgebra of $\A_\l(I)$ 
for all $I \in \I_0$.
\end{theorem}
We can summarise the main results of this section in the following theorem.

\begin{theorem}
Let $\l$ be a unitary, positive energy representation of the Neveu-Schwarz 
algebra with ${\rm Tr}(e^{-\beta L^\l_0 })<\infty$ for all $\beta > 0$. 
Then, with $\gA _\l(I)$ and $\delta_I$ as above, $I\in\I_0$, the triple 
$(\gA _\l, \H _\l, \delta_\l)$ is a net of generalised quantum algebras on $S^1\setminus\{-1\}$
with Hamiltonian $L^\l_0$.

In particular this applies if $\l$ is any irreducible unitary lowest 
weight representation of the Neveu-Schwarz algebra.
\end{theorem}
\begin{corollary}
With $\l$ as in the above theorem, $(\gA _\l, \H _\l, \delta_\l)$ extends to a 
rotation covariant net of generalised quantum algebras on the double cover $\S2$ 
of $S^1$ with Hamiltonian $L^\l_0$. It does not extends to a net on $S^1$.
\end{corollary}
\begin{proof}
According to the proof of Proposition \ref{prop:local-smoothness}, we have local 
rotation covariance, namely if $I$ and $e^{it}I$ belong to $\I_0$ for all 
$|t|<\varepsilon$ for some $\varepsilon >0$, 
then
\[e^{itL^\l_0}\gA_\l(I)e^{-itL^\l_0}=\gA_\l(e^{it}I)\]
and 
\[
\delta_{e^{it}I}= \Ad e^{itL^\l_0}\circ\delta_I\circ \Ad e^{-itL^\l_0}\ .
\]
Since 
$$e^{i4\pi L^\l_0} \in \bigcap_{I\in \I}\A_\l(I)',$$ 
the above equation and the group property of $t\mapsto e^{itL^\l_0}$ allow 
to extend consistently $(\gA _\l, \H _\l, \delta_\l)$ to a rotation covariant 
net of generalised quantum algebras on $\S2$.

Since 
$$e^{-i2\pi L^\l_0}\Gamma_\l \in \bigcap_{I\in \I}\A_\l(I)'$$ 
in the Neveu-Schwarz case, we have
\[
\Ad e^{i2\pi L^\l_0}\cdot\delta_I\cdot\Ad e^{-i2\pi L^\l_0} = \Ad\Gamma_\l 
\cdot\delta_I\cdot\Ad\Gamma_\l = \gamma_\l\cdot\delta_I\cdot\gamma_\l 
= -\delta_I \ ,
\]
namely the derivation $\delta_I$ associated with an interval $I\in\I^{(2)}$ 
changes sign after a $2\pi$-rotation, so it cannot give rise to  a net of 
generalised quantum algebras on $S^1$.
\end{proof}

\begin{remark}
In the Ramond case we found no obstruction to define the net 
$(\gA _\l, \H _\l, Q_\l)$ on $S^1$. This is due to the fact that 
if $\l$ is a representation of the Ramond algebras
then the unitary operator $e^{i2\pi L^\l_0}$ commutes with all the local 
algebras and hence it does not implement the grading $\gamma_\l$.
\end{remark}


\section{Outlook}\label{Outlook}
By the results in this paper, we have the basis for the analysis of the JLO 
cyclic cocycle and index theorems. One point to further describe is 
a ``universal algebra'' whose representations give rise to the spectral triples 
(in this paper we have worked on the representation space from the beginning).
Furthermore, there are different models, e.g. the supersymmetric free field.
This kind of issues and analysis will be the subject of subsequent work.

\bigskip

\noindent {{\small {\bf Acknowledgements.}
We thank Mih\'aly Weiner for useful discussions. Part of this work has been 
done while the authors were visiting the Erwin Schr\"odinger Institute in 
Vienna for the program on ``Operator Algebras and Conformal Field Theory" in 2008 and we gratefully acknowledge the hospitality there received.


{\footnotesize \begin{thebibliography}{99}

\bibitem{BR1} O. Bratteli \& D. W. Robinson, 
``Operator Algebras and Quantum Statistical Mechanics 1",
Springer-Verlag (1987).

\bibitem{BG} D. Buchholz \& H. Grundling,
{\it Algebraic supersymmetry: A case study},  Commun. Math. Phys. 
{\bf 272}, 699-750 (2007).

\bibitem{BS-M} D. Buchholz \& H. Schulz-Mirbach, 
{\it Haag duality in conformal quantum field theory}, Rev. Math. Phys. 
{\bf 2}, 105--125 (1990).

\bibitem{CKL} S. Carpi, Y. Kawahigashi \& R. Longo, 
{\it Structure and classification of superconformal nets},  
 Ann. Henri Poincar\'e. {\bf 9}, 1069--1121 (2008).

\bibitem{Connes91} A. Connes, {\it On the Chern character of $\theta$ 
summable Fredholm modules}, Commun. Math. Phys. {\bf 139}, 171--181 
(1991).

\bibitem{C} A. Connes, 
``Noncommutative Geometry'' Academic Press (1994).

\bibitem{CM} A. Connes \& M. Marcolli, ``Noncommutative Geometry, 
Quantum Fields and Motives" Preliminary version. www.alainconnes.org.   

\bibitem{FH} C. J. Fewster \& S. Hollands, {\it Quantum energy 
inequalities in two-dimensional conformal field theory}, Rev. Math. Phys. 
{\bf 17}, 577--612 (2005).

\bibitem{FQS} D. Friedan, Z. Qiu \& S. Shenker, 
{\it Superconformal invariance in two dimensions and the tricritical 
Ising model}, Phys. Lett. B {\bf 151}, 37--43 (1985).

\bibitem{GS} E. Getzler \& A. Szenes, {\it On the Chern character of 
a theta-summable Fredholm module}, J. Funct. Anal. {\bf 84}, 343--357
(1989).

\bibitem{GoWa} R. Goodman \& N. R. Wallach, {\it Projective unitary
positive-energy representations of ${\rm Diff}(S^1)$}, 
J. Funct. Anal. {\bf 63}, 299--321 (1985).

\bibitem{H} R. Haag, 
``Local Quantum Physics", Springer-Verlag (1996). 

 \bibitem{JLO} A. Jaffe, A. Lesniewski \& K. Osterwalder, 
{\it Quantum K-theory I. The Chern character}, Commun. Math. Phys. 
{\bf 118}, 1--14 (1988). 

 \bibitem{JLW} A. Jaffe, A. Lesniewski \& J. Weitsman,
{\it Index of a family of Dirac operators on loop space},
Commun. Math. Phys. {\bf 112}, 75--88 (1987).

\bibitem{K} D. Kastler, {\it Cyclic cocycles from graded KMS 
functionals}, Commun. Math. Phys. {\bf 121}, 345-350 (1989).

\bibitem{KL1}
Y. Kawahigashi \& R. Longo,
{\it Classification of local conformal nets. Case $c<1$},
Ann.\ of Math.\ {\bf 160}, 493--522 (2004).

\bibitem{L4} R. Longo, 
{\it Notes for a quantum index theorem}, 
Commun. Math. Phys. {\bf 222}, 45--96 (2001).

\bibitem{L5} R. Longo, 
{\it Index of subfactors and statistics of 
quantum fields. I}, Commun. Math. Phys. {\bf 126}, 217--247 (1989). 

\bibitem{Tol99} V. Toledano Laredo, {\it Integrating unitary 
representations of infinite-dimensional Lie groups},
J. Funct. Anal. {\bf 161}, 478--508 (1999).

\bibitem{X6}
F. Xu, {\it Mirror extensions of local nets},
Commun. Math. Phys. {\bf 270}, 835-847 (2007).

\end{thebibliography}}
\end{document}